\documentclass[11pt,leqno]{amsart}
\usepackage{amssymb,latexsym}

\usepackage{amsmath}
\usepackage{amsthm}
\usepackage{rotating}
\usepackage{graphicx}
\usepackage{latexsym}
\usepackage{amsfonts}
\usepackage{graphics}
\usepackage{color}
\usepackage{eufrak}
\usepackage{amstext}
\usepackage{amsopn}
\usepackage{amsbsy}
\usepackage{amscd}
\usepackage{amsxtra}
\usepackage{enumerate}
\usepackage{hyperref}
\usepackage{titletoc}

\title{On the Ambrosetti-Malchiodi-Ni Conjecture for general submanifolds}
\author{Fethi Mahmoudi }
\address{\noindent Fethi Mahmoudi - Departamento de Ingenier\'{\i}a Matem\'atica and
CMM, Universidad de Chile, Casilla 170 Correo 3, Santiago,
Chile.}\email{fmahmoudi@dim.uchile.cl}

\author{Felipe Subiabre S\'{a}nchez}
\address{\noindent Felipe Subiabre S\'{a}nchez - Departamento de Ingenier\'{\i}a  Matem\'atica and CMM, Universidad de Chile, Casilla 170 Correo 3,
Santiago, Chile}
\email{fsstol@gmail.com}

\author{Wei Yao }
\address{\noindent Wei Yao - Departamento de Ingenier\'{\i}a Matem\'atica and
CMM, Universidad de Chile, Casilla 170 Correo 3, Santiago,
Chile.}\email{wyao.cn@gmail.com}

\date{}

\newtheorem{theorem}{Theorem}[section]
\newtheorem{proposition}{Proposition}[section]

\newtheorem{lemma}{Lemma}[section]
\newtheorem{definition}{Definition}[section]
\newtheorem{remark}{Remark}[section]
\newcommand{\D}{\Delta}

\newcommand{\R}{\mathbb{R}}
\newcommand{\N}{\mathbb{N}}

\newcommand{\e}{\varepsilon}
\newcommand{\del}{\partial}

\newcommand{\ov}{\bar}

\newcommand{\G}{\Gamma}

\numberwithin{equation}{section}

\begin{document}

\maketitle
\begin{center}{\bf Abstract}\end{center}
We study  positive solutions of the following semilinear equation $$\e^2\Delta_{\bar g} u - V(z) u+
u^{p} =0\,\hbox{ on }\,M, $$ where $(M, \bar g )$ is a compact smooth $n$-dimensional Riemannian manifold without boundary or the Euclidean space $\R^n$,  $\e$ is a small positive  parameter, $p>1$ and $V$ is a uniformly positive smooth potential. Given $k=1,\dots,n-1$, and $1 < p < \frac{n+2-k}{n-2-k}$. Assuming that $K$ is  a
$k$-dimensional smooth, embedded compact submanifold of $M$, which is stationary and
non-degenerate with respect to the functional $\int_K V^{\frac{p+1}{p-1}-\frac{n-k}{2}}dvol$, we prove the existence of a sequence $\e=\e_j\to 0$ and  positive solutions $u_\e$ that concentrate along 
$K$. This result proves in particular the validity of a conjecture by Ambrosetti-Malchiodi-Ni \cite{amn1}, extending a recent result  by Wang-Wei-Yang \cite{wwy}, where the one co-dimensional case has been considered. Furthermore, our approach explores a connection between solutions of the nonlinear Schr\"{o}dinger equation and $f$-minimal submanifolds in manifolds with density.

\


\

\tableofcontents

\section{Introduction and  main results}

In this paper we study concentration phenomena for positive solutions of the nonlinear elliptic problem
\begin{equation}\label{eq:pe}
-\e^2 \D_{\bar{g}} u + V(z) u=|u|^{p-1}u\     \hbox{ on } M,
\end{equation}
where $M$ is an $n$-dimensional compact Riemannian manifold without boundary (or the flat Euclidean space $\R^n$), $\D_{\bar{g}}$ stands for the
Laplace-Beltrami operator on $(M,\bar {g})$, $V$ is a smooth positive function on $M$
satisfying
\begin{equation}\label{cond-V}
    0 < V_1 \leq V(z) \leq V_2,\quad\text{for all}\ z\in M\ \text{and for some constants}\ V_1,V_2,
\end{equation}
$u$ is a real-valued function, $\e > 0$ is a small parameter and $p$
is an exponent greater than one.

The above semilinear elliptic problem arises from the standing waves for the nonlinear Schr\"{o}dinger equation on $M$, see \cite{amn1,dkw} and some references therein  for more details. 
An interesting case is the {\em semiclassical limit} $\e \to
0$. For results in this direction, when $M=\R$ and $p=3$, Floer-Weinstein \cite{FW}  first proved the existence of solutions highly concentrated near critical points of $V$. Later on this result was extended by Oh \cite{Oh} to
$\R^n$ with $1<p<\frac{n+2}{n-2}$. More precisely, the profile of these solutions is given by the {\it ground state} $U_{V(x_0)}$  of the limit equation
\begin{equation}\label{eq:vx0}
-\Delta u+V(x_0)u-u^p=0\ \hbox{ in }\  \R^n,
\end{equation}
where $x_0$ is the concentration point. That is, the solutions obtained in \cite{FW} and \cite{Oh} behave qualitatively like
$$u_\e(x)\sim
U_{V(x_0)}(\frac{x-x_0}{\e}), \quad  \hbox{as $\e$ tends to zero}.
$$
Since $U_{V(x_0)}$ decays exponentially to $0$ at infinity, $u_\e$ vanishes rapidly away from $x_0$. In other words, in the semiclassical limit, solutions constructed in \cite{FW,Oh} concentrate at  points and they are always called {\em peak solutions} or spike solutions. In recent years, these existence results have been generalized in different directions, including: multiple peaks solutions, degenerate potentials, potentials tending to zero at infinity and for more general nonlinearities. {\it An important and interesting question is whether solutions exhibiting concentration on higher dimensional sets exist.}

Only recently it has been proven the existence of
solutions concentrating at higher dimensional sets, like curves or
spheres. In all these results (except for \cite{dap}), the
profile is given by (real) solutions to \eqref{eq:vx0} which are
independent of some of the variables. If concentration occurs near a $k$-dimensional set,
then the profile in the directions orthogonal to the limit set (concentration set) will be given by a soliton in $\R^{n-k}$. For example, some first results  in the case of radial symmetry were obtained by Badiale-D'Aprile
\cite{BAD} and  Benci-D'Aprile \cite{BED}.  These results  were improved by  Ambrosetti-Malchiodi-Ni \cite{amn1}, where necessary
and sufficient conditions for the location of the concentration set have been given.
Unlike the point concentration case, the limit set is not stationary for the potential $V$ : in fact a solution concentrated near a sphere carries a {\it potential energy} due to $V$ and a {\it volume energy}. Define 
\begin{equation}
E(u)=\frac{\e^2}{2}\int_{M}|\nabla_{\bar g}u|^2+V(z)u^2-\frac{1}{p+1}\int_M|u|^{p+1}
\end{equation}
and let $K$ be a $k$-dimensional submanifold of $M$ and $U_K$ be a proper approximate solution concentrated along $K$, see \eqref{glob-aprox} below. One has
\[
E(U_K)\sim\e^{n-k}\int_K V^{\theta_k} dvol, \qquad \hbox{with }\qquad \theta_k=\frac{p+1}{p-1}-\frac{1}{2}(n-k).
\]

Based on the above energy considerations, Ambrosetti-Malchiodi-Ni \cite{amn1} conjectured that
{\it concentration on $k$-dimensional sets for $k =1,\cdots,n-1$ is
expected}  under suitable non-degeneracy assumptions and the limit set $K$ should satisfy
\begin{equation}
\theta_k \nabla^N V=V{\bf H},
\end{equation}
where $\nabla^N$ is the normal gradient to $K$ and ${\bf H}$ is the mean-curvature vector on $K$. In particular, they suspected that concentration occurs in general along sequences $\e_j\rightarrow0$.

\

By developing an infinite dimensional version of the Lyapunov-Schmidt reduction method, del Pino-Kowalczyk-Wei \cite{dkw} successfully proved the validity of the above conjecture for $n= 2$ and $k =1$. Actually they proved that: given a non-degenerate stationary curve $K$ in $\R^2$ (for the weighted length functional
$\int_K V^{\frac{p+1}{p-1}-\frac{1}{2}}$), suppose that $\e$ is sufficiently small and satisfies the following {\it gap condition}:
\[
|\e^2\ell^2-\mu_0|\ge c\,\e,\quad \forall\,\ell \in\mathbb N,
\]
where $\mu_0$ is a fixed positive constant, then problem  \eqref{eq:pe} has a positive solution $u_\e$ which concentrates on $K$, in the sense that it is exponentially small away from $K$. After some time Mahmoudi-Malchiodi-Montenegro in \cite{mmm-CPAM} constructed a different type of solutions. Indeed, they studied complex-valued solutions whose phase is highly oscillatory carrying a quantum mechanical momentum along the limit curve. In particular they
established the validity of the above conjecture for the case $n\ge 2$ arbitrary and $k=1$. Recently, by applying the method developed in \cite{dkw}, Wang-Wei-Yang \cite{wwy} considered the one-codimensional case $n\ge 3$ and $k=n-1$ in the flat Euclidean space $\mathbb{R}^n$. The main purpose of this paper is to prove the validity of the above conjecture for all $k=1,\dots,n-1$. 

\

To prove the validity of the Ambrosetti-Malchiodi-Ni conjecture for all cases, one possible way is to  generalize the method developed in \cite{dkw} and \cite{wwy}. For this purpose, we first recall the key steps in \cite{dkw} and \cite{wwy}. According to our knowledge, the first key step is the construction of proper approximate solutions, and the second key step is to develop an infinite dimensional Lyapunov-Schmidt reduction method so that the original problem can be reduced to a simpler one that we can handle easily. Actually this kind of infinite dimensional reduction argument has been used in many constructions in PDE and geometric analysis. It has been developed by many authors working on this subject or on closely related problems, see for example \cite{dkw,dmm, GMP, mmp,mm1} and references therein. 

\

Let us now go back to our problem. To construct proper approximate solutions for general submanifolds, we first  expand the Laplace-Betrami operator for arbitrary submanifolds, see Proposition~\ref{expansion-g}. Then by an iterative scheme of Picard«s type, a family of very accurate approximate solutions can be obtained, see Section 3. Next we develop  an infinite dimensional reduction such that the construction of positive solutions of problem \eqref{eq:pe} can be reduced to the solvability of a reduced system~\eqref{reduce system-final}. For more details about the setting-up of the problem, we refer the reader to Subsection 4.1. It is slightly different from the arguments in \cite{dkw} and \cite{wwy}. Finally, by noticing the recent development on manifolds with density in differential geometry (cf. e.g. \cite{Liu-2012,morgan}), our method explores a connection between solutions of the nonlinear Schr\"{o}dinger equation and $f$-minimal submanifolds in Riemannian  manifolds with density. 

\

We are  now in position to state our main result. 

\begin{theorem}\label{th:main}
Let $M$ be a compact $n$-dimensional Riemannian manifold (or the  Euclidean space $\R^n$) and let $V : M \to \R$
be a smooth positive function satisfying \eqref{cond-V}. Given $k=1,\dots,n-1$, and $1 < p < \frac{n+2-k}{n-2-k}$. Suppose that $K$ be a stationary non-degenerate smooth compact submanifold in $M$ for the weighted functional
\begin{equation*}
\int_K V^{\frac{p+1}{p-1}-\frac{n-k}{2}} dvol,
\end{equation*}
then there is a sequence $\e_j \to 0$ such that problem \eqref{eq:pe}
possesses positive solutions $u_{\e_j}$ which  concentrate near $K$. Moreover, for some constants $C$, $c_0>0$, the solutions $u_{\e_j}$ satisfies globally
\begin{equation*}
|u_{\e_j}(z)|\leq C\exp\big(-c_0\,\text{dist}(z,K)\big/\e_j\big).
\end{equation*}

\end{theorem}

\begin{remark}
The assumptions on $K$ are  related to the existence of non-degenerate compact minimal submanifold in manifolds $M$ with density $V^{\frac{p+1}{p-1}-\frac{n-k}{2}}dvol$. In fact writing $V^{\frac{p+1}{p-1}-\frac{n-k}{2}}=e^{-f}$,
then $K$ is called $f$-minimal submanifold in differential geometry (cf. \cite{Liu-2012}).

\end{remark}

\begin{remark}
Actually we can prove that the same result holds true under a gap condition on $\e$, which  is due to a resonance phenomena. Similar conditions can be found in \cite{dkw,wwy} and some references therein.
\end{remark}

Before closing this introduction, we notice that problem \eqref{eq:pe} is similar to the following singular perturbation problem
  \begin{equation}\label{spp}
  \begin{cases}
    -\e^2 \D u + u = u^p & \text{ in } \Omega, \\
    \frac{\partial u}{\partial \nu} = 0 & \text{ on } \partial \Omega,
    \\ u > 0 & \text{ in } \Omega.
  \end{cases}
\end{equation}
This latter problem arises in the study of some biological models
 and as \eqref{eq:pe} it exhibits concentration of
solutions at some points of $\overline{\Omega}$. Since this equation is homogeneous, then  the location of concentration points is determined by the geometry of the domain. On the other hand, it has been proven that solutions exhibiting concentration on higher dimensional sets exist. For results in this direction we refer the reader to \cite{dmm,mm1,mm,mal,malm,malm2,wy}. 

\

In general, these results can be divided into two types: The first one is the case where the concentration set lies totally on the boundary. 
The second one is where the concentration set is inside the domain and which intersect the boundary transversally. For this second type of solutions we refer the reader to  Wei-Yang \cite{wy}, who proved the existence of layer on the line intersecting with the
boundary of a two-dimensional domain orthogonally. See also Ao-Musso-Wei \cite{amw}, where triple junction solutions have been constructed. In the over-mentioned two results, \cite{amw} and \cite{wy}, only the one dimensional concentration case has been considered. We believe the method developed here  to the above problem \eqref{spp} can be used to handle the higher dimensional situation, namely concentration at  arbitrary dimensional submanifolds which intersect the boundary transversally. Interestingly, our preliminary result shows that our method explores a connection between solutions of problem \eqref{spp} and minimal submanifolds with free boundary in geometric analysis. 

It is worth pointing out that \cite{wy} applied an infinite dimensional reduction method while \cite{amw} used a finite dimensional one. We also suggest the interested readers to the paper \cite{dwy} for an intermediate reduction method which can be interpreted as an intermediate procedure between the finite and the infinite dimensional ones. Moreover, it is interesting to consider Open Question 4 in \cite{dwy}, which can be seen as the Ambrosetti-Malchiodi-Ni Conjecture without the small parameter $\e$.

\

The paper is organized as follows.
In Section 2 we introduce the Fermi coordinates in a tubular neighborhood of $K$ in $M$ and we expand the Laplace-Beltrami operator in these Fermi coordinates. In Section 3, a family of very accurate approximate solutions is constructed. Section 4 will be devoted to develop an infinite dimensional Lyapunov-Schmidt reduction and to prove Theorem \ref{th:main}.




%
%

\section{Geometric background}

In this section we will give some geometric background. In particular, we will introduce the so-called Fermi coordinates which play important role in the higher dimensional concentrations. Before doing this, we first introduce the auxiliary weighted functional corresponding to problem \eqref{eq:pe}.

\subsection{The auxiliary weighted functional}

Let $K$ be a $k$-dimensional closed (embedded or immersed) submanifold of $M^n$, $1\le k\le n-1$. Let $\{K_t\}_t$ be a smooth one-parameter family of  submanifolds such that $K_0=K$. We  define
\begin{equation}
\mathcal{E}(t)=\int_{K_t}V^{\sigma} dvol,
\qquad \text{ with }\quad \  \sigma=\frac{p+1}{p-1}-\frac{n-k}{2}.
\end{equation}
Denote $\nabla^T$ and $\nabla^N$ to be connections projected to the
tangential and normal spaces on  $K$. We give the following definitions on $K$ which  appeared in Theorem~\ref{th:main}. 

\begin{definition} [Stationary condition]
A submanifold $K$ is said to be stationary relative to the functional $\int_K V^{\sigma} dvol$ if
\begin{align}
\sigma\nabla^NV=-VH\ \text{on}\ K,
\end{align}
where $H$ is the mean curvature vector on $K$, i.e., $H_j=-\Gamma_{aj}^a$ (here the minus sign depends on the orientation, and $\Gamma_a^b$ are the 1-forms on the normal bundle of $K$ (see \eqref{eq:Gab} below for the definition).
\end{definition}

\begin{definition} [Nondegeneracy (ND) condition]
We say that $K$ is non-degenerate if the quadratic form
\begin{align}
&\int_K\Bigg\{\Big\langle\Delta_K\Phi+\frac{\sigma}{V}\nabla_KV\cdot\nabla_K\Phi,\Phi\Big\rangle+\sigma^{-1}H(\Phi)^2-\frac{\sigma}{V}\big(\nabla^N\big)^2V\,[\Phi,\Phi]
-{\rm Ric}(\Phi,\Phi)\nonumber\\
&\qquad\qquad\qquad\qquad\qquad\qquad+\Gamma_{b}^a(\Phi)\Gamma_{a}^b(\Phi)\Bigg \}V^\sigma\sqrt{\det(g)}\, dvol
\end{align}
defined on the normal bundle to $K$,
%
is non-degenerate.
\end{definition}

\begin{remark}
Here and in the rest of this paper, Einstein summation convention is used, that is, summation over repeated
indices is understood. 

\

If we set $V^\sigma=e^{-f}$, i.e., $f=-\sigma\ln V$, then our stationary and ND conditions are corresponding to the first and second variation formulas of $f$-minimal submanifold in \cite{Liu-2012}, i.e.,
\begin{equation*}
H=\nabla^Nf,
\end{equation*}
where $H=-\sum_{a}\nabla^N_{e_a}e_a$ is the mean curvature vector, $e_a$ ($1\leq a\leq k$) is an orthonormal frame in an open set of $K$. And at $t=0$, 
\begin{align*}
\frac{d^2}{dt^2}\bigg(  \int_{K_t}e^{-f}\bigg)
=&
\int_{K}e^{-f}\bigg(-\sum_{a=1}^kR_{avva}-\frac{1}{2}\Delta_K(|v|^2)+|\nabla_Kv|^2-2|A^v|^2-f_{vv}\\
&\qquad\qquad+\frac{1}{2}\langle\nabla^Tf,\nabla^T(|v|^2)\rangle\bigg),
\end{align*}
where $K_t$ is a smooth family of  submanifolds such that $K_0=K$, the variational normal vector field $v$ is compactly supported on $K_t$, and $A^v_{ab}=-\langle\nabla_{e_a}e_b,v\rangle$.
\end{remark}

%
%

\subsection{Fermi coordinates and expansion of the metric}\label{ss:fc}
Let $K$ be a $k$-dimensional submanifold of $(M,\ov g)$
($1\le k\le n-1$). Define $N=n-k$, we choose
along $K$ a local orthonormal frame field $\big((E_a)_{a=1,\cdots,
k},(E_i)_{i=1,\cdots, N}\big)$ which is oriented. At points of $K$, we have the natural splitting
$$TM=T K \oplus N K$$ where $T K$ is the
tangent space to $K$ and $N K$ represents the normal bundle, which
are spanned respectively by $(E_a)_a$ and $(E_i)_i$.

We denote by $\nabla$ the connection induced by the metric $\ov{g}$ and by
$\nabla^N$ the corresponding normal connection on the normal bundle.
Given $p \in K$, we use some geodesic coordinates $y$ centered at
$p$. We also assume that at $p$ the normal vectors $(E_i)_i$, $i =
1, \dots, N$, are transported parallely (with respect to $\nabla^N$)
through geodesics from $p$, so in particular
\begin{equation}\label{eq:parall}
    \ov g\left(\nabla_{E_a}E_j\,,E_i\right)=0  \ \hbox{ at } p,\quad \forall\, i,j = 1, \dots, N,\ a = 1, \dots, k.
\end{equation}
In a neighborhood of $p$ in $K$, we consider normal geodesic
coordinates
\[
f(\bar{y}) : = \exp^K_p (y_a E_a), \quad  \forall\,\bar{y} := (y_{1}, \ldots, y_{k}),
\]
where $\exp^K$ is the exponential map on $K$ and summation over repeated
indices is understood. This yields the coordinate vector fields
$X_a : = f_* (\del_{y_a})$. We extend the $E_i$ along each geodesic $\gamma_E(s)$ so that they are parallel
with respect to the induced connection on the normal bundle $NK$.
This yields an orthonormal frame field $X_i$ for $NK$ in a neighborhood of
$p$ in $K$ which satisfies
\[
\left. \nabla_{X_a} X_i \right|_p \in T_p K.
\]

A coordinate system in a neighborhood of $p$ in $M$ is now defined by
\begin{equation}\label{eqF}
F(\bar{y},\bar x) := \exp^{M}_{f(\bar y)}( x_i \, X_i), \quad
\forall\,(\bar y,\bar x) :=(y_{1}, \ldots, y_{k},x_1, \ldots, x_{N}),
\end{equation}
with corresponding coordinate vector fields
\[
X_i : = F_* (\del_{x_i}) \quad \mbox{and} \quad  X_a : = F_*
(\del_{y_a}).
\]


By our choice of coordinates, on $K$ the metric $\ov{g}$ splits in
the following way
\begin{equation}\label{eq:splitovg}
    \ov g(q) = \ov g_{ab}(q)\,d y_a\otimes d y_b+\ov
g_{ij}(q)\,dx_i\otimes dx_j, \quad \forall q \in K.
\end{equation}
We denote by $\Gamma_a^b(\cdot)$ the 1-forms defined on the normal
bundle, $NK$, of $K$ by the formula~
\begin{equation}\label{eq:Gab}
\ov g_{bc} \Gamma_{ai}^c:=  \ov g_{bc} \Gamma_a^c(X_i)=\ov g(\nabla_{X_a}X_b,X_i) \quad \hbox{at } q=f(\bar y).
\end{equation}
Notice that
\begin{equation}
\label{eq:min}
    K \hbox{ is minimal } 
    \quad \Longleftrightarrow \quad \sum_{a=1}^k\G^a_a(E_i)
    = 0 \quad \hbox{ for any } i = 1, \dots N.
\end{equation}

Define $q=f(\bar y)=F(\bar y,0)\in K$ and let $(\widetilde g_{ab}(y))$ be the induced metric on $K$. When we consider the metric coefficients in a neighborhood of $K$,
we obtain a deviation from formula \eqref{eq:splitovg}, which is
expressed by the next lemma. 
We will denote by $R_{\alpha\beta\gamma\delta}$ the components of the curvature tensor
with lowered indices, which are obtained by means of the usual ones
$R_{\beta\gamma\delta}^\sigma$ by~
\begin{equation}
\label{ctens} R_{\alpha\beta\gamma\delta}=\ov
g_{\alpha\sigma}\,R_{\beta\gamma\delta}^\sigma.\end{equation}


\begin{lemma} At the point $F(\bar y,\bar x)$, the following expansions hold,
for any $a=1,...,k$ and any $i,j=1,...,N$, where $N=n-k$, 
\begin{align*}
\bar g_{ij}&=\delta_{ij}+\frac{1}{3}\,R_{istj}\,\bar x_s\,\bar x_t\,
+\,{\mathcal O}(|\bar x|^3);\\[3mm]
\bar g_{aj}&=\frac23\widetilde{g}_{ab}R^b_{kjl}\bar x^k\bar x^l+
{\mathcal O}(|\bar x|^3);\\[3mm]
\bar g_{ab}&=\widetilde{g}_{ab}-\Big\{\widetilde{g}_{ac}\,\Gamma_{bi}^c+\widetilde{g}_{bc}\,\Gamma_{ai}^c\Big\}\,\bar x_i
+\Big[R_{sabl}+\widetilde g_{cd}\Gamma_{as}^c\,\Gamma_{bl}^d \Big]\bar x_s \bar x_l+\,{\mathcal O}(|\bar x|^3).
\end{align*}
Here $R_{istj}$  are  computed at the point of
$K$ parameterized by $(\bar y,0)$. \label{lemovg}
\end{lemma}

\begin{proof}
The proof is somewhat standard and is thus omitted, we refer to \cite{dmm} for details, see also Proposition 2.1 in \cite{mmp}.
\end{proof}


By the Whitney embedding theorem, $K\subset M\hookrightarrow\R^{2n}$. Thus we can define $K_\e:=K/\e$ and $M_\e:=M/\e$ in a natural way. On the other hand since  $F(\bar y,\bar x)$ is a Fermi coordinate system on $M$, then $F_\e(y,x):=F(\e y,\e x)/\e$ defines a Fermi coordinate system on $M/\e$. With this notation, here and in the sequel, by slight abuse of notation we denote $V(\e y,\e x)$ to actually mean $V(\e z)=V\big(F(\e y,\e x)\big)$ in the Fermi coordinate system. The same way is understood to its derivatives with respect to $y$ and $x$.

Now we can introduce our first parameter function $\Phi$ which is a normal vector field defined on $K$ and define $x=\xi+\Phi(\e y)$. Then $(y,\xi)$ is the Fermi coordinate system for the submanifold $K_\Phi$. Adjusting the parameter $\Phi$, later we will show that there are solutions concentrating on $K_\Phi$ for a subsequence of $\e$. 

We denote by $ g_{\alpha\beta}$ the metric coefficients in the new coordinates $(y,\xi)$. It follows that
$$
g_{\alpha\beta}=\sum\limits_{\gamma,\delta}\bar g_{\gamma\delta}\,\frac{\partial{z_\alpha}}{\partial \xi_\gamma}\,\frac{\partial{z_\beta}}{\partial \xi_\delta}.
$$
Which yields
$$
 g_{ij}=\bar g_{ij}|_{\xi+\Phi}, \qquad  g_{aj}=\bar g_{aj}|_{\xi+\Phi}+\e\,\partial_{\bar a}\Phi^l \bar g_{jl}|_{\xi+\Phi},
$$
and
$$
g_{ab}=\bar g_{ab}|_{\xi+\Phi}+\e\,\Big\{ \bar g_{aj}\,\partial_{\bar b} \Phi^j+\bar g_{bj}\,\partial_{\bar a} \Phi^j\Big\}|_{\xi+\Phi}+\e^2\,\partial_{\bar a} \Phi^i\,\partial_{\bar b} \Phi^j\,\bar g_{ij}|_{\xi+\Phi}
$$
where summations over repeated indices is understood.

To express the error terms, it is convenient to introduce some notations. For a positive integer $q$, we denote by $R_q(\xi)$, $R_q(\xi,\Phi)$, $R_q(\xi,\Phi,\nabla\Phi)$, and $R_q(\xi,\Phi,\nabla\Phi,\nabla^2\Phi)$ error terms such that the following bounds hold for some positive constants $C$ and $d$:
\begin{align*}
|R_q(\xi)|\leq C\varepsilon^q(1+|\xi|^d),
\end{align*}
\begin{align*}
&|R_q(\xi,\Phi)|\leq C\varepsilon^q(1+|\xi|^d),\\
&|R_q(\xi,\Phi)-R_q(\xi,\bar{\Phi})|\leq C\varepsilon^q(1+|\xi|^d)|\Phi-\bar{\Phi}|,
\end{align*}
\begin{align*}
&|R_q(\xi,\Phi,\nabla\Phi)|\leq C\varepsilon^q(1+|\xi|^d),\\
&|R_q(\xi,\Phi,\nabla\Phi)-R_q(\xi,\bar{\Phi},\nabla\bar{\Phi})|\leq C\varepsilon^q(1+|\xi|^d)\big(|\Phi-\bar{\Phi}|+|\nabla\Phi-\nabla\bar{\Phi}|\big),
\end{align*}
and
\begin{align*}
|R_q(\xi,\Phi,\nabla\Phi,\nabla^2\Phi)|\leq& C\varepsilon^q(1+|\xi|^d)+C\varepsilon^{q+1}(1+|\xi|^d)|\nabla^2\Phi|,
\end{align*}
\begin{align*}
&\big|R_q(\xi,\Phi,\nabla\Phi,\nabla^2\Phi)-R_q(\xi,\bar{\Phi},\nabla\bar{\Phi},\nabla^2\bar{\Phi})\big|\\
&\leq C\varepsilon^q(1+|\xi|^d)\big(|\Phi-\bar{\Phi}|+|\nabla\Phi-\nabla\bar{\Phi}|\big)\big(1+\varepsilon|\nabla^2\Phi|+\varepsilon|\nabla^2\bar{\Phi}|\big)\\
&\quad+C\varepsilon^{q+1}(1+|\xi|^d)|\nabla^2\Phi-\nabla^2\bar{\Phi}|.
\end{align*}

Using the expansion of the previous lemma, one can easily show that the following lemma holds true.

\begin{lemma}
In the coordinate $(y,\xi)$, the metric coefficients satisfy
\begin{align*}
{g}_{ab}&=\widetilde g_{ab}-\varepsilon\,\big\{\widetilde g_{bf}\Gamma_{ak}^f+\widetilde g_{af}\Gamma_{bk}^f\big\}\,(\xi^k+\Phi^k)+\varepsilon^2\big(R_{kabl}+\widetilde g_{cd}\,\Gamma_{ak}^c\Gamma_{bl}^d\big)(\xi^k+\Phi^k)(\xi^l+\Phi^l)\\[3mm]
&\quad+\varepsilon^2\partial_{\bar a}\Phi^j\partial_{\bar b}\Phi^j+R_3(\xi,\Phi,\nabla\Phi),\\
{g}_{aj}&=\varepsilon\partial_{\bar a}\Phi^j+\frac{2}{3}\varepsilon^2R_{kajl}(\xi^k+\Phi^k)(\xi^l+\Phi^l)+R_3(\xi,\Phi,\nabla\Phi),\\[3mm]
{g}_{ij}&=\delta_{ij}+\frac{1}{3}\varepsilon^2R_{kijl}(\xi^k+\Phi^k)(\xi^l+\Phi^l)+R_3(\xi,\Phi,\nabla\Phi).
\end{align*}
\end{lemma}

Denote the inverse metric of $({g}_{\alpha\beta})$ by $({g}^{\alpha\beta})$. Recall that, given the expansion of a matrix as $M=I+\varepsilon A+\varepsilon^2B+\mathcal{O}(\varepsilon^3)$, we have
\begin{align*}
M^{-1}=I-\varepsilon A-\varepsilon^2B+\varepsilon^2A^2+\mathcal{O}(\varepsilon^3).
\end{align*}

\begin{lemma}\label{lem:inverse-det}
In the coordinate $(y,\xi)$, the metric coefficients ${g}^{\alpha\beta}$ satisfy
\begin{align*}
{g}^{ab}&=\widetilde g^{ab}+\e\,\bigg\{\widetilde g^{cb}\,\Gamma_{ci}^a+\widetilde g^{ca}\,\Gamma_{ci}^b\bigg\}\,(\xi^i+\Phi^i)-\e^2\,\widetilde g^{cb}\,\widetilde g^{ad}\,R_{kcdl}\,(\xi^k+\Phi^k)(\xi^l+\Phi^l)\\
&\quad+\varepsilon^2\bigg(\widetilde g^{ac}\,\Gamma_{dk}^b\Gamma_{cl}^d+\widetilde g^{bc}\,\Gamma_{dk}^a \Gamma_{cl}^d
+\widetilde g^{cd}\,\Gamma_{dk}^a \Gamma_{cl}^b\bigg)\,(\xi^k+\Phi^k)(\xi^l+\Phi^l)+R_3(\xi,\Phi,\nabla\Phi),\\[3mm]
{g}^{aj}&=-\varepsilon\,\widetilde g^{ab}\,\partial_{\bar b}\Phi^j-\frac{2\,\e^2}{3}R_{kajl}(\xi^k+\Phi^k)(\xi^l+\Phi^l)+
\varepsilon^2\partial_{\bar b}\Phi^j\,\bigg\{\widetilde g^{bc}\,\Gamma_{ci}^a+\widetilde g^{ac}\,\Gamma_{ci}^b\bigg\}\,(\xi^i+\Phi^i)\\
&\quad+R_3(\xi,\Phi,\nabla\Phi),\\[3mm]
{g}^{ij}&=\delta_{ij}-\frac{\e^2}{3}\,R_{kijl}(\xi^k+\Phi^k)(\xi^l+\Phi^l)+\varepsilon^2\,\widetilde g^{ab}\,\partial_{\bar a}\Phi^i\partial_{\bar b}\Phi^j+R_3(\xi,\Phi,\nabla\Phi).
\end{align*}
Furthermore, we have the validity of the following expansion for the
log of the determinant of $g$:
\begin{align*}
 \log\big(\det g\big)& =  \log\big(\det \widetilde g\big)-2\e\,\Gamma^b_{bk}\,(\xi^k+\Phi^k)+  \frac13\,\e^2\, R_{mssl}\,(\xi^m+\Phi^m) \,(\xi^l+\Phi^l)\\
 &\quad+
\e^2\, \Big( \widetilde g^{ab}\,R_{mabl}-\Gamma_{am}^{c}
\Gamma_{cl}^{a}  \Big)\,(\xi^m+\Phi^m) \,(\xi^l+\Phi^l)+R_3(\xi,\Phi,\nabla\Phi).\\
\end{align*}
\end{lemma}
\begin{proof}  The expansions of the  metric in the above lemma follow from Lemma \ref{lemovg}  while the expansion of the log of the determinant of $g$ follows from the fact that one can write  $g=G+M$ with
\[
G=\bigg(\begin{matrix}
  \widetilde g & 0 \\
  0& Id_{\R^N}
 \end{matrix}\bigg)
 \quad \hbox{and}\quad M =\mathcal{O}(\e),
\]
then we have the following expansion
\[
\log\big(\det g\big)=\log\big(\det G\big)+{\rm tr}(G^{-1}M)-\frac12 {\rm tr}\Big((G^{-1}M)^2\Big)+\mathcal{O}(\|M\|^3).
\]
and the lemma follows at once.
\end{proof}
\subsection{Expansion of the Laplace-Beltrami operator}
In terms the above  notations, we have the following expansion of the Laplace-Beltrami operator.
\begin{proposition}\label{expansion-g}
Let $u$ be a smooth function on $M_\e$. Then in the Fermi coordinate $(y,\xi)$, we have that
\begin{align*}
\Delta_{g}u&=\partial_{ii}^2u+\D_{K_\e}u-\varepsilon\,\Gamma_{bj}^b\partial_ju-2\varepsilon\,\widetilde g^{ab}\,\partial_{\bar b}\Phi^j\,\partial_{aj}^2u+2\,\e\,\widetilde g^{cb}\,\Gamma_{cs}^a\,(\xi^s+\Phi^s)\partial_{ab}^2u\\
&\quad+\varepsilon^2\,\nabla_K\Phi^i\cdot \nabla_K\Phi^j\,\partial_{ij}^2u-\frac{1}{3}\varepsilon^2R_{kijl}(\xi^k+\Phi^k)(\xi^l+\Phi^l)\partial_{ij}^2u-\e^2\,\Gamma^d_{dk}\,\partial_{\bar b}\Phi^k\,\widetilde g^{ab}\partial_au
\\
&\quad-\frac{4}{3}\varepsilon^2R_{kajl}(\xi^k+\Phi^k)(\xi^l+\Phi^l)\partial_{aj}^2u+2\varepsilon^2\partial_{\bar b}\Phi^j\,\bigg\{\widetilde g^{bc}\,\Gamma_{ci}^a+\widetilde g^{ac}\,\Gamma_{ci}^b\bigg\}\,(\xi^i+\Phi^i)\,\partial_{aj}^2u\\
&\quad+\e^2\,\bigg\{-\widetilde g^{cb}\,\widetilde g^{ad}\,R_{kcdl}+\widetilde g^{ac}\,\Gamma_{dk}^b\Gamma_{cl}^d+\widetilde g^{bc}\,\Gamma_{dk}^a \Gamma_{cl}^d+\widetilde g^{cd}\,\Gamma_{dk}^a \Gamma_{cl}^b\bigg\}\,(\xi^k+\Phi^k)(\xi^l+\Phi^l)
\,\partial_{ab}^2u\\
&\quad+\varepsilon^2\bigg(\widetilde g^{ab}\,R_{kabj}+\frac{2}{3}R_{kiij}-\Gamma_{ak}^c\Gamma_{cj}^a\bigg)(\xi^k+\Phi^k)\partial_ju
-\varepsilon^2\D_K\Phi^j\partial_ju\\
&\quad+2\varepsilon^3\partial_{\bar a\bar b}^2\Phi^j\Gamma_{ak}^b(\xi^k+\Phi^k)\partial_ju\\
&\quad-\e^2\,\bigg(\widetilde g^{ab}\,\partial_{\bar a}\Gamma_{dk}^d-\partial_{\bar a}\big\{\widetilde g^{cb}\Gamma_{ck}^a+\widetilde g^{ca}\Gamma_{ck}^b\big\}\bigg)\,(\xi^k+\Phi^k)\partial_bu-\frac{2}{3}\varepsilon^2R_{jajk}(\xi^k+\Phi^k)\partial_au\\
&\quad+2\e^2\,\bigg\{\widetilde g^{cb}\,\Gamma_{ci}^a+\widetilde g^{ca}\,\Gamma_{ci}^b\bigg\}\,\partial_{\bar b}\,\Phi^i\,\partial_au+\frac{1}{2}\,\e^2\,\partial_{\bar a}(\log\det \widetilde{g})\,
\big\{\widetilde g^{cb}\Gamma_{ci}^a+\widetilde g^{ca}\Gamma_{ci}^b\big\}(\xi^i+\Phi^i)\partial_bu\\
&\quad+R_3(\xi,\Phi,\nabla\Phi,\nabla^2\Phi)(\partial_ju+\partial_au)+R_3(\xi,\Phi,\nabla\Phi)(\partial_{ij}^2u+\partial_{aj}^2u+\partial_{ab}^2u).
\end{align*}
\end{proposition}

%


\begin{remark}\label{remark2.2}
The proof of Proposition~\ref{expansion-g} will be postponed to the Appendix. It is worth mentioning that the coefficients of all the derivatives of $u$ in the above expansion are smooth bounded functions of the variable $\bar y=\e y$. The slow dependence of theses coefficients of $y$  is important in our construction of some proper approximate solutions. 
\end{remark}

\section{Construction of approximate solutions}

To prove Theorem~\ref{th:main}, the first key step in our method is to construct some proper approximate solutions. To achieve this goal, we have introduced some geometric background, especially the Fermi coordinates. The main objective of this section is to construct some very accurate local approximate solutions in a tubular neighbourhood of $K_\e$ by an iterative scheme of Picard's type and to define some proper global approximate solutions by the gluing method.

\subsection{Facts on the limit equation}

Recall that by the scaling, equation~\eqref{eq:pe} becomes
\begin{align}\label{eq-u}
\Delta_{g}u-V(\varepsilon z)u+u^p=0.
\end{align}
In the Fermi coordinate $(y,x)$, we can write $V(\e z)=V(\e y,\e x)$. Taking $x=\xi+\Phi(\e y)$, we have the following expansion of potential:
\begin{align}
V(\varepsilon y,\varepsilon x)=V(\e y,0)+\varepsilon\langle\nabla^NV(\e y,0),\xi+\Phi\rangle+\frac{\varepsilon^2}{2}(\nabla^N)^2V(\e y,0)[\xi+\Phi]^2+R_3(\xi,\Phi).
\end{align}
If the profile of solutions depends only on $\xi$ or varies slower on $y$, by the expansion of the Laplace-Beltrami operator in Proposition~\ref{expansion-g} and the above expansion of potential, the leading equation is
\begin{equation}\label{leadingeqn}
\sum_{i=1}^N\partial_{\xi_i\xi_i}^2u-V(\e y, 0)u+u^p=0.
\end{equation}

Define 
\begin{equation}
\mu(\e y)=V(\e y, 0)^{1/2},\quad
h(\e y)=V(\e y, 0)^{1/(p-1)},\quad\forall\,y\in K_\e.
\end{equation}
For the leading equation~\eqref{leadingeqn}, by the scaling
\[
u(y,\xi)=h(\e y)v\big(\mu(\e y) \xi\big)
=h(\e y)v(\bar \xi),
\]
the function $v$ satisfies
\begin{equation}\label{lim-equ}
\Delta_{\mathbb{R}^N}v-v+v^p=0.
\end{equation}
We call this equation {\em the limit equation}.

We now turn to the equation~\eqref{eq-u}, in the spirit of above argument, we look for a solution $u$ of the form
\begin{equation}\label{eq:uv}
u(y,\xi)=h(\e y)v\big(y,\bar\xi\big) 
\quad\text{with}\ \bar \xi=\mu(\e y)\xi\in\mathbb{R}^N.
\end{equation}
An easy computation shows that 
\begin{align*}
\partial_{a}u&= h\,\partial_a v
+\e(\partial_{\bar a} h)v+\e\,h\,\partial_{\bar a}\mu\,\xi^j\partial_jv,\\
\partial^2_{ij}u&= h\,\mu^2\,\partial^2_{ij}v,\\
\partial^2_{aj}u&= \e\,\Big( \mu \partial_{\bar a}h +h \partial_{\bar a}\mu\Big) \,\partial_{j}v+\,h\,\mu\,\partial^2_{a j}v+\e\,h\,\mu\,\xi^i\,\partial_{\bar a}\mu\,\partial^2_{i j}v,\\
\partial^2_{ab}u&= h\,\partial^2_{ab} v+\e\,\Big( \partial_{\bar b}h \,\partial_a v+\partial_{\bar a}h \,\partial_b v+h \partial_{\bar b}\mu\,\xi^j\partial^2_{aj}v+h \partial_{\bar a}\mu\,\xi^j\partial^2_{bj}v\Big)\\
&\quad+\e^2\Big(\partial_{\bar a}h\partial_{\bar b}\mu \xi^j \partial_j v+\partial_{\bar b}h\partial_{\bar a}\mu \xi^j \partial_j v+\partial_{\bar a\bar b}^2h v+h\partial_{\bar a}\mu\partial_{\bar b}\mu\xi^i\xi^j\partial_{ij}^2v
+h\partial_{\bar a\bar b}^2\mu \xi^j\partial_jv\Big),
\end{align*}
and 
\begin{align*}
\D_{K_\e}u&=\e^2 \Delta_{K}h\, v+h\,\D_{K_\e}v+2\e\,\nabla_{K} h\cdot \nabla_{K_\e} v
+\e^2\,\big(h\,\D_K \mu+2\,\nabla_K h\cdot \nabla_K\mu\big) \,\xi^j\,\partial_j v\nonumber\\
&\quad+\e^2\,h\, |\nabla_K\mu|^2\,\xi^j\xi^l \,\partial^2_{jl}v+2\e\,h\, \xi^j\,\nabla_K\mu\cdot (\nabla_{K_\e}\partial_{j}v).
\end{align*}

Therefore, we get the following expansion of the Laplace-Beltrami operator on $u$:
\begin{align*}
h^{-1}\mu^{-2}\,\Delta_{g}u
&=\Delta_{\mathbb{R}^N}v+\mu^{-2}\,\D_{K_\e}v+B(v),
\end{align*}
with
$
B(v)=B_1(v)+B_2(v).
$
Where $B_j$'s are respectively given by
\begin{align*}
B_1(v)
=&-\e\,\mu^{-1}\,\Gamma_{bj}^b\,\partial_jv
+\varepsilon^2\,\mu^{-1}\,\Big(\widetilde g^{ab}\,R_{kabj}+\frac{2}{3}R_{kiij}-\Gamma_{ak}^c\Gamma_{cj}^a\Big)(\frac1\mu\bar\xi^k+\Phi^k)\partial_jv
\\
&+\e^2 \,h^{-1}\,\mu^{-2}\,\Delta_{K}h\, v+2\e^2\,(h\,\mu^2)^{-1}\,\nabla_K h\cdot\Big( \frac{\bar \xi^j}{\mu}\,\nabla_K\mu -\mu\,\nabla_K\Phi^j\Big)\,\partial_jv\\
&+2\e\,\,h^{-1}\,\mu^{-2}\,\nabla_{K} h\cdot \nabla_{K_\e} v-\frac{1}{3}\e^2\,R_{kijl}(\frac1\mu\bar\xi^k+\Phi^k)(\frac1\mu\bar\xi^l+\Phi^l)\partial_{ij}^2v\\
&+\e^2\,\Big(\mu^{-2}\bar\xi^i\,\nabla_K\mu-\nabla_K\Phi^i  \Big)\,\Big( \mu^{-2}\bar\xi^j\,\nabla_K\mu-\nabla_K\Phi^j  \Big)\,\partial_{ij}^2v\\
&+\e^2\,\,\mu^{-2}\,\Big(\frac{\bar\xi^j}{\mu}\,\D_K \mu-2\,\nabla_K\mu\cdot\nabla_K\Phi^j-\mu\,\D_K\Phi^j\Big) \,\partial_j v\\
&+2\e\,\mu^{-2}\,\Big( \frac{\bar\xi^j}{\mu}\,\nabla_K\mu-\mu\,\,\nabla_K\Phi^j\Big)\cdot\nabla_{K_\e}\big(\partial_j v\big),
\end{align*}
and
\begin{align*}
&h\mu^2B_2(v)
=-\e^2\,h\,\Gamma^d_{dj}\,\nabla_K\Phi^j\,\cdot \nabla_{K_\e}v\\
&+2\,\e\,\widetilde g^{cb}\,\Gamma_{cs}^a\,\Big(\frac1\mu\,\bar\xi^s+\Phi^s\Big)\,\Big( h\,\partial^2_{ab} v+\e\,\Big\{ \partial_{\bar b}h \,\partial_a v+\partial_{\bar a}h \,\partial_b v+h \partial_{\bar b}\mu\,\frac{\bar \xi^j}{\mu}\partial^2_{aj}v+h \partial_{\bar a}\mu\,\frac{\bar \xi^j}{\mu}\,\partial^2_{bj}v\Big\}  \Big)\\
&-\frac{4}{3}\varepsilon^2\,h\,\mu\,R_{kajl}\Big(\frac1\mu\bar\xi^k+\Phi^k\Big)\Big(\frac1\mu\bar\xi^l+\Phi^l\Big)\partial_{aj}^2v+2\varepsilon^2\,h\,\mu\,\partial_b\Phi^j\,\Big\{\widetilde g^{bc}\,\Gamma_{ci}^a+\widetilde g^{ac}\,\Gamma_{ci}^b\Big\}\,\Big(\frac1\mu\bar\xi^i+\Phi^i\Big)\,\partial_{aj}^2v\\
&+\e^2\,h\,\Big\{-\widetilde g^{cb}\,\widetilde g^{ad}\,R_{kcdl}+2\widetilde g^{ac}\,\Gamma_{dk}^b\Gamma_{cl}^d+\widetilde g^{cd}\,\Gamma_{dk}^a \Gamma_{cl}^b\Big\}\,\Big(\frac1\mu\bar\xi^k+\Phi^k\Big)\Big(\frac1\mu\bar\xi^l+\Phi^l\Big)
\,\partial_{ab}^2v\\
&+2\varepsilon^3\,h\,\mu\,\partial_{\bar a\bar b}^2\Phi^j\Gamma_{ak}^b\Big(\frac1\mu\bar\xi^k+\Phi^k\Big)\partial_jv\\
&-\e^2\,h\,\Big(\widetilde g^{ab}\,\partial_{\bar a}\Gamma_{dk}^d-\partial_{\bar a}\Big\{\widetilde g^{cb}\Gamma_{ck}^a+\widetilde g^{ca}\Gamma_{ck}^b\Big\}\Big)\,\Big(\frac1\mu\bar\xi^k+\Phi^k\Big)\partial_bv
-\frac{2}{3}\varepsilon^2\,h\,R_{jajk}\Big(\frac1\mu\bar\xi^k+\Phi^k\Big)\partial_av\\
&+2\e^2\,h\,\Big\{\widetilde g^{cb}\,\Gamma_{ci}^a+\widetilde g^{ca}\,\Gamma_{ci}^b\Big\}\,\partial_{\bar b}\,\Phi^i\,\partial_av
+\frac{1}{2}\,\e^2\,h\,\partial_{\bar a}(\log\det \widetilde{g})\,
\Big\{\widetilde g^{cb}\Gamma_{ci}^a+\widetilde g^{ca}\Gamma_{ci}^b\Big\}\Big(\frac1\mu\bar\xi^i+\Phi^i\Big)\partial_bv\\
&+R_3(\xi,\Phi,\nabla\Phi,\nabla^2\Phi)\Big(\partial_jv+\partial_av\Big)+R_3(\xi,\Phi,\nabla\Phi)\Big(\partial_{ij}^2v+\partial_{aj}^2v+\partial_{ab}^2v\Big).\\
\end{align*}

Setting 
$$
S_\e(u)=-\Delta_{g}u+V(\e z)u-u^p,
$$
then by using the above expansion we can write
\begin{align*}
h^{-1}\mu^{-2}\,S_\e(u)
&=-\Delta_{\mathbb{R}^N}v-\mu^{-2}\,\D_{K_\e}v-B(v)+\mu^{-2}\,V(\e z)v-h^{p-1}\mu^{-2}\,v^p\\
&=-\Delta_{\mathbb{R}^N}v+v-v^p-\mu^{-2}\,\D_{K_\e}v+\mu^{-2}\,\Big(V(\e y,\e x)-V(\e y,0) \Big)\,v-B(v).
\end{align*}
Now using the following expansion of potential:
\begin{align*}
V(\varepsilon y,\varepsilon x)=V(\e y,0)+\varepsilon\langle\nabla^NV(\e y,0),\frac{\bar \xi}{\mu}+\Phi\rangle+\frac{\varepsilon^2}{2}(\nabla^N)^2V(\e y,0)[\frac{\bar \xi}{\mu}+\Phi]^2+R_3(\bar \xi,\Phi),
\end{align*}
we obtain
\begin{align}\label{equiv-u-v}
h^{-1}\mu^{-2}\,S_\e(u)=-\Delta_{\mathbb{R}^N}v+v-v^p-\mu^{-2}\,\D_{K_\e}v-\widetilde B(v)=:\widetilde S_\e(v),
\end{align}
where $\widetilde B(v)=\widetilde B_1(v)+\widetilde B_2(v)$ with
$$
\widetilde B_1(v)=B_1(v)-\mu^{-2}\,\Big(\e\langle\nabla^NV(\e y,0),\frac{\bar \xi}{\mu}+\Phi\rangle+\frac{\varepsilon^2}{2}(\nabla^N)^2V(\e y,0)[\frac{\bar \xi}{\mu}+\Phi]^2\Big)\,v
$$
and
$$
\widetilde B_2(v)= B_2(v)-R_3(\bar \xi,\Phi)\,v.
$$

At the end of this subsection, let us list some basic and useful properties of positive solutions of the limit equation~\eqref{lim-equ}. 
\begin{proposition}\label{prop-w0}
If $1<p<\infty$ for $N=2$ and $1<p<\frac{N+2}{N-2}$ for $N\geq3$, then every solution of problem:
\begin{align}
\left\{
\begin{array}{ll}
-\Delta_{\mathbb{R}^N} v+v-v^p=0\ \text{in}\ \mathbb{R}^N,
\vspace{1mm}\\
v>0\ \text{in}\ \mathbb{R}^N,\ v\in H^1(\mathbb{R}^N),
\end{array}\right.
\end{align}has the form $w_0(\cdot-Q)$ for some $Q\in\mathbb{R}^N$, where $w_0(x)=w_0(|x|)\in C^\infty(\mathbb{R}^N)$ is the unique positive radial solution which satisfies
\begin{equation}
\lim_{r\rightarrow\infty}r^{\frac{N-1}{2}}e^rw_0(r)=c_{N,p},\quad
\lim_{r\rightarrow\infty}\frac{w_0'(r)}{w_0(r)}=-1.
\end{equation}
Here $c_{N,p}$ is a positive constant depending only on $N$ and $p$. Furthermore, $w_0$ is non-degenerate in the sense that
\[
\text{Ker}\left(-\Delta_{\mathbb{R}^N}+1-pw_0^{p-1}\right)\cap L^\infty(\mathbb{R}^N)=\text{Span}\Big\{\partial_{x_1}w_0,\cdots,\partial_{x_N}w_0\Big\},
\]
and the Morse index of $w_0$ is one, that is, the linear operator 
$$L_0:=-\Delta_{\mathbb{R}^N}+1-pw_0^{p-1}$$ has only one negative eigenvalue $\lambda_0<0$, and the unique even and positive eigenfunction corresponding to $\lambda_0$ can be denoted by $Z$. 
\end{proposition}

\begin{proof}
This result is well known. For the proof we refer the interested reader to \cite{BL-83} for the existence, \cite{GNN-81} for the symmetry, \cite{kwo} for the uniqueness, Appendix C in \cite{NT-93} for the nondegeneracy, and \cite{BW-10} for the Morse index.
\end{proof}

As a corollary, there is a constant $\gamma_0>0$ such that 
\begin{equation}\label{inequality-orth}
\int_{\mathbb{R}^{N}}\Big\{|\nabla\phi|^2+\phi^2-pw_0^{p-1}\phi^2\Big\}\,d\bar{\xi}\geq\gamma_0\int_{\mathbb{R}^{N}}\phi^2\,d\bar{\xi},
\end{equation}
whenever $\phi\in H^1(\mathbb{R}^N)$ and
\begin{align*}
\int_{\mathbb{R}^{N}}\phi\,\partial_jw_0\,d\bar{\xi}=0=\int_{\mathbb{R}^{N}}\phi Z\,d\bar{\xi},
\quad \forall\,j=1,\dots,N.
\end{align*}

\subsection{Local approximate solutions}
In a tubular neighbourhood of $K_\e$, \eqref{equiv-u-v} makes it obvious that $S_\e(u)=0$ is equivalent to $\widetilde S_\e(v)=0$. 

By the expression of $\widetilde S_\e(v)$ and Remark~\ref{remark2.2}, we look for approximate solutions of the form
\begin{align}\label{form-u}
v=v(y,\bar \xi)=w_0(\bar \xi)+\sum_{\ell=1}^I\e^\ell w_\ell(\e y,\bar\xi)+\e e(\e y) Z(\bar \xi),
\end{align}
where $I\in\mathbb{N}_+$, $w_0$ and $Z$ are given in Proposition~\ref{prop-w0}, $w_\ell$'s and $e$ are smooth bounded functions on their variables. 

The idea for introducing $eZ$ in \eqref{form-u} comes directly from \cite{dkw,wwy}. The reason is the linear theory in Section 4.2.2, especially Lemma~\ref{lemma-injectivity}.

To solve $\widetilde S_\e(v)=0$ accurately, the normal section $\Phi$ is to be chosen in the following form
$$
\Phi=\Phi_0+\sum_{\ell=1}^{I-1}\e^\ell \,\Phi_\ell,
$$
where $\Phi_0,\dots,\Phi_{I-1}$ are smooth bounded functions on $\bar y$.

\subsubsection{Expansion at first order in $\varepsilon$ : }

We first solve the equation $\widetilde S_\e(v)=0$ up to order $\varepsilon$. Here and in the following we will write $\mathcal{O}(\varepsilon^j)$ for terms that appear at the $j$-th order in an expansion.

Suppose $v$ has the form \eqref{form-u}, then
\begin{align*}
\widetilde S_\e(v)
&=\e\Big(-\D_{\R^N}w_1+w_1-pw_0^{p-1}w_1\Big)
+\e\big(-\e^2\mu^{-2}\D_K e+\lambda_0e\big)Z\\
&\quad+\varepsilon\Big(\mu^{-1}\Gamma_{bj}^b\partial_jw_0
+\mu^{-2}\,\langle\nabla^NV(\varepsilon y,0),\frac{\bar \xi}{\mu}+\Phi_0\rangle w_0\Big)
+\mathcal{O}(\varepsilon^2).
\end{align*}

Hence the term of order $\e$ in the right-hand side of above equation vanishes if and only if the function $w_1$ solves
\begin{equation}\label{eqw1}
L_0w_1=-\mu^{-1}\,\Gamma_{bj}^b\partial_jw_0-\mu^{-2}\,\langle\nabla^NV(\varepsilon y,0),\frac{\bar \xi}{\mu}+\Phi_0\rangle w_0.
\end{equation}

Here and in the following, we will keep the term $\e\big(-\e^2\mu^{-2}\D_K e+\lambda_0e\big)Z$ in the error. The reason is simply that it cannot be cancelled without solving an equation of $e$ since $L_0Z=\lambda_0Z$. 

By Proposition~\ref{prop-w0}, equation~\eqref{eqw1} is solvable if and only if for all $i=1,\dots,N$, 
\begin{align}\label{eqw1-cond}
\int_{\mathbb{R}^N}\Big(\mu^{-1}\,\Gamma_{bj}^b\partial_jw_0+\mu^{-2}\,\langle\nabla^NV(\varepsilon y,0),\frac{\bar \xi}{\mu}+\Phi_0\rangle\, w_0\Big)\partial_iw_0\,d\bar\xi=0.
\end{align}
Since $w_0$ is radially symmetric, \eqref{eqw1-cond} is equivalent to 
$$
\Gamma_{bi}^b\int_{\mathbb{R}^N}|\partial_1w_0|^2\,d\bar\xi=\frac{1}{2}\,\mu^{-2}\,\partial_i V(\e y,0)\,\int_{\mathbb{R}^N}w_0^2\,d\bar\xi.
$$
Recalling the identity
\begin{align}
\frac{1}{2}\int_{\mathbb{R}^N}w_0^2\,d\bar\xi=\sigma\int_{\mathbb{R}^N}|\partial_1w_0|^2\,d\bar\xi
\quad \hbox{with}\ 
\sigma=\frac{p+1}{p-1}-\frac{N}{2},
\end{align}
we get
\begin{align}\label{stationary}
\sigma\nabla^NV(\varepsilon y,0)=-V(\varepsilon y,0)H(\varepsilon y),
\end{align}
where $H=(-\Gamma_{bi}^b)_i$ is the mean curvature vector on $K$. This is exactly our stationary condition on $K$. 

When \eqref{stationary} holds, the equation of $w_1$ becomes
\begin{equation}\label{eqw1-new}
L_0w_1
=-\mu^{-1}\,\Gamma_{bj}^b\Big(\partial_jw_0+\sigma^{-1}\bar\xi^j w_0\Big)
+\sigma^{-1}\langle H,\Phi_0\rangle w_0.
\end{equation}
Hence we can write
\begin{align}
w_1=w_{1,1}+w_{1,2},
\end{align}
where
\begin{align}\label{eq:w1w2}
w_{1,1}=-\mu^{-1}\,\Gamma_{bj}^bU_j \quad\hbox{and}\quad
w_{1,2}=\sigma^{-1}\langle H,\Phi_0\rangle U_0.
\end{align}
Here $U_j$ is the unique smooth bounded function satisfying
\begin{align}\label{eq-w-j}
L_0U_j=\partial_jw_0+\sigma^{-1}\,\bar\xi^j\,w_0,\quad \int_{\mathbb{R}^N}U_j\,\partial_iw_0\,d\bar\xi=0,\ \forall\, i=1,\dots,N,
\end{align}
and $U_0$ is the unique smooth bounded function such that
\begin{align}\label{eq-U-0}
L_0U_0=w_0,\qquad \int_{\mathbb{R}^N}U_0\,\partial_iw_0\,d\bar\xi=0,\ \forall \,i=1,\dots,N.
\end{align}
It follows immediately that $w_1=w_1(\e y,\bar\xi)$ is smooth bounded on its variable. Furthermore, it is easily seen that $U_j$ is odd on variable $\bar\xi^j$ and is even on other variables. Moreover, $U_0$ has an explicit expression
\begin{align}\label{w-0}
U_0=-\frac{1}{p-1}w_0-\frac{1}{2}\bar\xi\cdot\nabla w_0.
\end{align}
\subsubsection{Expansion  at second order in $\varepsilon$}

In this subsection we will solve the equation $\widetilde S_\e(v)=0$ up to order $\varepsilon^2$ by solving $w_2$ and $\Phi_0$ together. 

Suppose $v$ has the form \eqref{form-u}, then
\begin{align*}
\widetilde S_\e(v)&=\e^2\Big(-\D_{\R^N}w_2+w_2-pw_0^{p-1}w_2\Big)
+\e \big( -\e^2\mu^{-2}\D_K e+\lambda_0e \big)Z\\
&\quad+\e^2\mathfrak{F}_2+\e^2\mathfrak{G}_2+\mathcal{O}(\e^3),
\end{align*}
where
\begin{align*}
\mathfrak{F}_2
=&\mu^{-1}\,\Gamma_{bj}^b\partial_jw_1+\mu^{-2}\,\langle\nabla^NV,\Phi_{1}\rangle \,w_{0}
+\frac{1}{3}\,R_{kijl}(\frac1\mu\bar\xi^k+\Phi_0^k)(\frac1\mu\bar\xi^l+\Phi_0^l)\partial_{ij}^2w_0\\
&-\mu^{-1}\Big(\widetilde g^{ab}\,R_{kabj}+\frac{2}{3}R_{kiij}-\Gamma_{ak}^c\Gamma_{cj}^a\Big)(\frac{\bar \xi^k}{\mu}+\Phi_0^k)\partial_jw_0\\
&-\mu^{-2}\Big(\frac{\bar\xi^j}{\mu}\,\D_K \mu-2\,\nabla_K\mu\cdot\nabla_K\Phi_0^j-\mu\,\D_K\Phi_0^j\Big) \,\partial_j w_0\\
&-h^{-1}\mu^{-2}\,\Delta_{K}h\,w_0
-2(h\mu^2)^{-1}\nabla_K h\cdot\Big( \frac{\bar \xi^j}{\mu}\,\nabla_K\mu -\mu\,\nabla_K\Phi_0^j\Big)\,\partial_jw_0\\
&-\Big(\mu^{-2}\bar\xi^i\,\nabla_K\mu-\nabla_K\Phi_0^i\Big)\Big( \mu^{-2}\bar\xi^j\,\nabla_K\mu-\nabla_K\Phi_0^j  \Big)\,\partial_{ij}^2w_0\\
&+\mu^{-2}\,\langle\nabla^NV,\frac{\bar \xi}{\mu}+\Phi_0\rangle w_1+\frac{1}{2}\,\mu^{-2}\,(\nabla^N)^2V[\frac{\bar \xi}{\mu}+\Phi_0,\frac{\bar \xi}{\mu}+\Phi_0]\,w_0-\frac{1}{2}p(p-1)w_0^{p-2}w_1^2,
\end{align*}
and 
\begin{align*}
\mathfrak{G}_2
=&\mu^{-1}\Gamma_{bj}^b\,e\,\partial_jZ
+\mu^{-2}\,\langle\nabla^NV,\frac{\bar \xi}{\mu}+\Phi_0\rangle eZ
-\frac{1}{2}p(p-1)w_0^{p-2}\Big\{(w_1+eZ)^2-w_1^2\Big\}.
\end{align*}

Hence the term of order $\e^2$ vanishes (except the term $\e\big(-\e^2\mu^{-2}\D_K e+\lambda_0e \big)Z$) if and only if $w_2$ satisfies the equation
\begin{align*}
L_0w_2=-\mathfrak{F}_2-\mathfrak{G}_2.
\end{align*}
By Freedholm alternative this equation is solvable if and only if $\mathfrak{F}_2+\mathfrak{G}_2$ is $L^2$ orthogonal to the kernel of linearized operator $L_0$, which is spanned by the functions $\partial_i w_0$, $i=1,\dots,N$.

It is convenient to write $\mathfrak{F}_2$ as
$$
\mathfrak{F}_2
=\mu^{-2}\langle\nabla^NV,\Phi_{1}\rangle w_{0}
+\widetilde{\mathfrak{F}}_2.
$$
Then $\widetilde{\mathfrak{F}}_2$ does not involve $\Phi_1$. By \eqref{stationary}, similar to $w_1$, we can write $w_2$ as
$$
w_2=w_{2,1}+w_{2,2},
$$
where 
$w_{2,2}=\sigma^{-1}\langle H,\Phi_1\rangle U_0$
solves the equation 
$$
L_0w_{2,2}=-\mu^{-2}\langle\nabla^NV,\Phi_{1}\rangle w_{0},
$$
and $w_{2,1}$ will solve the equation 
$$
L_0w_{2,1}=-\widetilde{\mathfrak{F}}_2-\mathfrak{G}_2.
$$

To solve the equation on $w_{2,1}$ we write  
\begin{equation*}
\widetilde{\mathfrak{F}}_2=\widetilde{\mathfrak{F}}_2(\Phi_0)=S_{2,0}+S_2(\Phi_0)+N_2(\Phi_0),
\end{equation*}
where $S_{2,0}=\widetilde{\mathfrak{F}}_2(0)$ does not involve $\Phi_0$, $S_2(\Phi_0)$ is the sum of linear terms of $\Phi_0$, and $N_2(\Phi_0)$ is the nonlinear term of $\Phi_0$.

Recall that  $w_1=w_{1,1}+w_{1,2}$ with
\begin{align*}
w_{1,1}=-\mu^{-1}\,\Gamma_{bj}^bU_j \quad\hbox{and}\quad
w_{1,2}=\sigma^{-1}\langle H,\Phi_0\rangle U_0.
\end{align*}
Then
\begin{align*}
S_{2,0}
=&\mu^{-1}\,\Gamma_{bj}^b\,\partial_jw_{1,1} 
+\frac{1}{3}\mu^{-2}R_{kijl}\,(\bar\xi^k\,\bar\xi^l\,\partial_{ij}^2w_0)
-\mu^{-2}\Big(\widetilde g^{ab}\,R_{kabj}+\frac{2}{3}R_{kiij}-\Gamma_{ak}^c\Gamma_{cj}^a\Big)(\bar \xi^k\,\partial_jw_0)\\
&-(\mu^{-3}\D_K \mu)(\bar\xi^j\,\partial_j w_0)
-(h^{-1}\mu^{-2}\Delta_{K}h)\, w_0
-2(h\mu^3)^{-1}(\nabla_K h\cdot\nabla_K\mu)(\bar \xi^j\,\partial_jw_0)\\
&-\mu^{-4}|\nabla_K\mu|^2\,(\bar\xi^i\,\bar\xi^j\,\partial_{ij}^2w_0)
+\mu^{-3}\langle\nabla^NV,\bar \xi\rangle w_{1,1}
+\frac{1}{2}\mu^{-4}(\nabla^N)^2V[\bar \xi,\bar \xi]\,w_0\\
&-\frac{1}{2}p(p-1)w_0^{p-2}w_{1,1}^2,
\end{align*}
\begin{align*}
S_2(\Phi_0)
=&\mu^{-1}\Gamma_{bj}^b\,\partial_jw_{1,2}
+\frac{2}{3}\mu^{-1}R_{kijl}\,\Phi_0^l\,(\bar\xi^k\,\partial_{ij}^2w_0)
-\mu^{-1}\Big(\widetilde g^{ab}\,R_{kabj}+\frac{2}{3}R_{kiij}-\Gamma_{ak}^c\Gamma_{cj}^a\Big)\Phi_0^k\,\partial_jw_0\\
&+\mu^{-2}\Big(2\,\nabla_K\mu\cdot\nabla_K\Phi_0^j+\mu\,\D_K\Phi_0^j\Big)\partial_j w_0
+2(h\mu)^{-1}\Big(\nabla_K h\cdot\nabla_K\Phi_0^j\Big)\partial_jw_0\\
&+2\mu^{-2}\Big(\nabla_K\mu\cdot\nabla_K\Phi_0^j\Big)\,(\bar\xi^i\,\partial_{ij}^2w_0)
+\mu^{-3}\langle\nabla^NV,\bar \xi\rangle w_{1,2}
+\mu^{-2}\langle\nabla^NV,\Phi_0\rangle w_{1,1}\\
&+\mu^{-3}(\nabla^N)^2V[\Phi_0,\bar \xi]\,w_0-p(p-1)w_0^{p-2}w_{1,1}w_{1,2},
\end{align*}
and 
\begin{align*}
N_2(\Phi_0)
&=\frac{1}{3}R_{kijl}\,\Phi_0^k\,\Phi_0^l\,\partial_{ij}^2w_0
-(\nabla_K\Phi_0^i\cdot\nabla_K\Phi_0^j)\,\partial_{ij}^2w_0
+\mu^{-2}\langle\nabla^NV,\Phi_0\rangle w_{1,2}\\
&\quad+\frac{1}{2}\,\mu^{-2}\,(\nabla^N)^2V[\Phi_0,\Phi_0]\,w_0
-\frac{1}{2}p(p-1)w_0^{p-2}w_{1,2}^2.
\end{align*}
Therefore, 
\begin{align*}
\int_{\R^N} S_2(\Phi_0)\,\partial_s w_0
=&\mu^{-1}\Gamma_{bj}^b\int_{\R^N}\partial_jw_{1,2}\,\partial_s w_0
+\frac{2}{3}\mu^{-1}R_{kijl}\,\Phi_0^l\int_{\R^N} \bar\xi^k\,\partial_{ij}^2w_0\,\partial_sw_0\\
&-\mu^{-1}\Big(\widetilde g^{ab}\,R_{kabj}+\frac{2}{3}R_{kiij}-\Gamma_{ak}^c\Gamma_{cj}^a\Big)\,\Phi_0^k\int_{\R^N}\partial_jw_0\,\partial_sw_0\\
&+\mu^{-2}\Big(2\nabla_K\mu\cdot\nabla_K\Phi_0^j+\mu\,\D_K\Phi_0^j\Big)\int_{\R^N}\partial_j w_0\,\partial_sw_0\\
&+2(h\mu)^{-1}\Big(\nabla_K h\cdot\nabla_K\Phi_0^j\Big)\int_{\R^N}\partial_jw_0\,\partial_s w_0\\
&+2\mu^{-2}\Big(\nabla_K\mu\cdot\nabla_K\Phi_0^j\Big)\int_{\R^N}\bar\xi^i\,\partial_{ij}^2w_0\,\partial_sw_0\\
&+\mu^{-2}\partial_jV(\e y,0)\,\Big(\mu^{-1}\int_{\R^N}\bar\xi^j\,w_{1,2}\,\partial_sw_0+ \Phi_0^j\int_{\R^N}w_{1,1}\,\partial_sw_0\Big)\\
&+\mu^{-3}\partial_{ij}^2V (\e y,0)\,\Phi_0^j\int_{\R^N}\bar\xi^i\,w_0 \,\partial_sw_0\\
&-p(p-1)\int_{\R^N} w_0^{p-2}\,w_{1,1}\,w_{1,2}\,\partial_sw_0.
\end{align*}

Let us denote by $A$ the sum of terms involving $w_{1,1}$ and $w_{1,2}$ in the above formula. Using \eqref{stationary} and \eqref{eq:w1w2} we can write
\begin{align*}
A=\mu^{-1}\,\sigma^{-1}\,\langle H,\Phi_0\rangle\Gamma_{aj}^a\int_{\mathbb{R}^N}\Big(\partial_jU_0+U_j+\sigma^{-1}\,\bar\xi^j\,U_0+p(p-1)w_0^{p-2}\,U_j\,U_0\Big\}\partial_sw_0.
\end{align*}
To compute this term we differentiate the equation \eqref{eq-w-j} on $U_j$ with respect to the variable $\bar\xi^j$ to obtain
\begin{align}
L_0(\partial_jU_j)-p(p-1)w_0^{p-2}U_j\partial_jw_0=\partial_{jj}^2w_0+\sigma^{-1}w_0+\sigma^{-1}\,\bar\xi^j\,\partial_jw_0.
\end{align}
Multiplying the above equation by $U_0$ and integrating by parts, we have
\begin{align*}
&\int_{\mathbb{R}^N}\Big\{\partial_jU_0+U_j+\sigma^{-1}\,\bar\xi^j\,U_0+p(p-1)w_0^{p-2}U_j\,U_0\Big\}\partial_jw_0\\
&=-\int_{\mathbb{R}^N}\big(2\partial_{jj}^2w_0+\sigma^{-1}\,w_0\big)U_0\\
&=-2\int_{\mathbb{R}^N}\big(-\frac{1}{p-1}w_0-\frac{1}{2}\,\bar\xi^l\partial_lw_0\big)\partial_{jj}^2w_0
-\sigma^{-1}\int_{\mathbb{R}^N}\big(-\frac{1}{p-1}w_0-\frac{1}{2}\,\bar\xi^l\partial_lw_0\big)\,w_0\\
&=-\big(\frac{2}{p-1}+1-\frac{N}{2}\big)\int_{\mathbb{R}^N}|\partial_1w_0|^2
-\sigma^{-1}\big(\frac{N}{4}-\frac{1}{p-1}\big)\int_{\mathbb{R}^N}w_0^2\\
&=-\int_{\mathbb{R}^N}|\partial_1w_0|^2.
\end{align*}

On the other hand, by direct computations we have
\begin{align*}
\int_{\mathbb{R}^N}\partial_jw_0\,\partial_sw_0
=\delta_{js}\int_{\mathbb{R}^N}(\partial_1w_0)^2,
\end{align*}
\begin{align*}
\int_{\mathbb{R}^N}\partial_{kj}^2w_0\,\bar\xi^k\partial_sw_0
=\frac{1}{2}\,\delta_{js}\int_{\mathbb{R}^N}\bar\xi^k\partial_k(\partial_jw_0)^2
=-\frac{N}{2}\,\delta_{js}\int_{\mathbb{R}^N}(\partial_1w_0)^2,
\end{align*}
$$
R_{kijl}\,\Phi_0^l\int_{\R^N} \bar\xi^k\,\partial_{ij}^2w_0\,\partial_sw_0
=R_{sjjl}\,\Phi_0^l\int_{\mathbb{R}^N}(\partial_1w_0)^2,
$$
$$
\Big(\widetilde g^{ab}\,R_{kabj}+\frac{2}{3}R_{kiij}-\Gamma_{ak}^c\Gamma_{cj}^a\Big)\,\Phi_0^k\int_{\R^N}\partial_jw_0\,\partial_sw_0
=
\Big(\widetilde g^{ab}\,R_{kabs}+\frac{2}{3}R_{kiis}-\Gamma_{ak}^c\Gamma_{cs}^a\Big)\,\Phi_0^k\int_{\mathbb{R}^N}(\partial_1w_0)^2.
$$
Summarizing, we have
\begin{align*}
\int_{\R^N} S_2(\Phi_0)\,\partial_s w_0
&=\mu^{-1}\bigg\{  
\D_K\Phi_0^s-\Big(\widetilde g^{ab}\,R_{kabs}-\Gamma_{ak}^c\Gamma_{cs}^a\Big)\Phi_0^k
+(2-N)\mu^{-1}\nabla_K\mu\cdot\nabla_K\Phi_0^s
\\
&\quad+2h^{-1}\nabla_Kh\cdot\nabla_K\Phi_0^s-\sigma\mu^{-2}\,\partial^2_{sj}V(\e y,0)\Phi_0^j-\sigma^{-1}\Gamma_{as}^a\langle H,\Phi_0\rangle\bigg\}\int_{\mathbb{R}^N}(\partial_1w_0)^2.
\end{align*}
Now, using the fact that
$$
\mu^{-1}\,\nabla_K\mu=\frac12 \,V^{-1}\,\nabla_K V \quad  \hbox{and}\quad h^{-1}\,\nabla_Kh=\frac{1}{p-1} \,V^{-1}\,\nabla_K V,
$$
we obtain (recalling the definition of  $\sigma$) that
$$
(2-N)\mu^{-1}\nabla_K\mu\cdot\nabla_K\Phi_0^s+2h^{-1}\nabla_Kh\cdot\nabla_K\Phi_0^s=\sigma\,V^{-1}\,\nabla_K V \cdot\nabla_K\Phi_0^s.
$$
Hence we summarize
\begin{align*}
\int_{\R^N} S_2(\Phi_0)\,\partial_s w_0
&=\mu^{-1}\bigg\{  
\D_K\Phi_0^s-\Big(\widetilde g^{ab}\,R_{kabs}-\Gamma_{ak}^c\Gamma_{cs}^a\Big)\Phi_0^k
+
\sigma V^{-1}\nabla_K V \cdot\nabla_K\Phi_0^s\\
&\qquad\qquad
-\sigma\mu^{-2}\partial^2_{sj}V(\e y,0)\Phi_0^j+\sigma^{-1}\Gamma_{bj}^b\,\Gamma_{as}^a\,\Phi_0^j\bigg\}\int_{\mathbb{R}^N}|\partial_1w_0|^2.
\end{align*}

Define $\mathcal{J}_K:NK\mapsto NK$ is a linear operator from the family of smooth sections of normal bundle to $K$ into itself, whose components are given by
\begin{align}\label{J}
\begin{array}{ll}
(\mathcal{J}_K\Phi_0)^s
=\D_K\Phi_0^s
-\Big(\widetilde g^{ab}\,R_{kabs}-\Gamma_{ak}^c\Gamma_{cs}^a\Big)\Phi_0^k
+\sigma V^{-1}\nabla_K V \cdot\nabla_K\Phi_0^s
\vspace{2mm}\\
\quad\qquad\qquad
-\sigma\mu^{-2}\partial^2_{sj}V(\bar y,0)\Phi_0^j+\sigma^{-1}\Gamma_{bj}^b\,\Gamma_{as}^a\,\Phi_0^j.
\end{array}
\end{align}
Then
\begin{align}
\int_{\R^N} S_2(\Phi_0)\,\partial_s w_0
&=\mu^{-1}\big(\int_{\mathbb{R}^N}|\partial_1w_0|^2\big)(\mathcal{J}_K\Phi_0)^s(\e y).
\end{align}

On the other hand, it is easy to check that 
\begin{align}
\int_{\R^N} S_{2,0}\,\partial_s w_0=0=\int_{\R^N} N_2(\Phi_0)\,\partial_s w_0
\end{align}
and 
\begin{align*}
\int_{\R^N}\mathfrak{G}_2\,\partial_s w_0
&=\bigg\{\mu^{-1}\Gamma_{bs}^b\int_{\R^N}\partial_s Z\,\partial_s w_0
+\mu^{-3}\partial_sV(\e y,0)\int_{\R^N}\bar\xi^s\,Z\,\partial_s w_0\nonumber\\
&\qquad-p(p-1)\int_{\R^N}w_0^{p-2}w_{1,1}Z\,\partial_s w_0\bigg\}e\\
&=\mu^{-1}\Gamma_{bs}^be
\int_{\R^N}\Big\{\partial_s Z+\sigma^{-1}Z\,\bar\xi^s+p(p-1)w_0^{p-2}Z\,U_s\Big\}\partial_s w_0\\
&=c_0\mu^{-1}\Gamma_{bs}^be.
\end{align*}

Therefore, the solvability of equation on $w_2$ is equivalent to the solvability of following equation on $\Phi_0$:
\begin{equation}\label{eq-phi0}
\mathcal{J}_K\Phi_0=\mathfrak H_2(\bar y;e),
\end{equation}
where $\mathfrak H_2(\bar y;e)=c_0He$ is a smooth bounded function. 

By the non-degeneracy condition on $K$, \eqref{eq-phi0} is solvable. Moreover, for any given $e$, it is easy to check that $\Phi_0=\Phi_0(\bar y;e)$ is a smooth bounded function on $\bar y$ and is Lipschitz continuous with respect to $e$. 

Now let us go back  to the equation of $w_{2,1}$:
$$
L_0w_{2,1}=-\widetilde{\mathfrak{F}}_2-\mathfrak{G}_2.
$$
Since both $\widetilde{\mathfrak{F}}_2$ and $\mathfrak{G}_2$ are smooth bounded functions of $(\e y,\bar\xi)$. Hence $w_{2,1}=w_{2,1}(\e y,\bar\xi)$ is also a smooth bounded function of $(\e y,\bar\xi)$. Moreover, $w_{2,1}=w_{2,1}(\e y,\bar\xi;e)$ is Lipschitz continuous with respect to $e$. 

\subsubsection{Higher order approximations}  

The construction of higher order terms follows exactly from the same calculation. Indeed,  to solve the equation up to an error of order $\e^{j+1}$ for some $j\ge 3$, we use an iterative scheme of Picard's type : assuming all the functions $w_i$'s ($1\le i\le j-1$) constructed,  we need to choose a function $w_j$ to solve an equation similar to that of $w_2$ (with obvious modifications) by solving an equation of $\Phi_{j-2}$ similar to that of $\Phi_0$.

When we collect all terms  of order $\mathcal{O}(\e^j)$ in $\widetilde S_\e(v)$, assuming all $w_i$'s for $i=1,\cdots j-1$ constructed (by the iterative scheme), we have 
\begin{align*}
\widetilde S_\e(v)
&=\e^j\Big(-\D_{\R^N}w_j+w_j-pw_0^{p-1}w_j\Big)
+\e\big( -\e^2\mu^{-2}\D_K e+\lambda_0e \big)Z\\
&\quad+\e^j\mathfrak{F}_j
+\e^j\mathfrak{E}_j\,e\,Z
+\e^j\mathcal A_j^i(\e y,\bar\xi ; \Phi_0,\cdots,\Phi_{j-3})\,e\,\partial_i Z\\
&\quad+\e^j\mathcal B_j^{i\ell}(\e y,\bar\xi ; \Phi_0,\cdots,\Phi_{j-3})\,e\,\partial^2_{i\ell} Z
+\e^j\mathcal C_j^i(\e y,\bar\xi ; \Phi_0,\cdots,\Phi_{j-3})\cdot \nabla_Ke\,\partial_i Z\\
&\quad+\e^j\mathcal D_j^{ab}(\e y,\bar\xi ; \Phi_0,\cdots,\Phi_{j-3})\,\partial^2_{ab}e\,Z
+\mathcal{O}(\e^{j+1}),
\end{align*}
with
\begin{align*}
\mathfrak{F}_j
=&\mu^{-1}\Gamma_{bl}^b\,\partial_lw_{j-1}
+\frac{2}{3}\mu^{-1}R_{kisl}\,\bar\xi^k\,\Phi_{j-2}^l\,\partial_{is}^2w_0
-\mu^{-1}\Big(\widetilde g^{ab}\,R_{kabs}+\frac{2}{3}R_{kiis}-\Gamma_{ak}^c\Gamma_{cs}^a\Big)\Phi_{j-2}^k\,\partial_sw_0\\
&+\mu^{-2}\Big(2\,\nabla_K\mu\cdot\nabla_K\Phi_{j-2}^s+\mu\,\D_K\Phi_{j-2}^s\Big)\partial_s w_0
+2(h\mu)^{-1}\Big(\nabla_K h\cdot\,\nabla_K\Phi_{j-2}^s\Big)\partial_sw_0\\
&+2\mu^{-2}\Big(\nabla_K\mu\cdot\nabla_K\Phi_{j-2}^s\Big)(\bar\xi^i\,\partial_{is}^2w_0)
+\mu^{-2}\langle\nabla^NV,\Phi_0\rangle w_{j-1}
+\mu^{-2}\langle\nabla^NV,\Phi_{j-2}\rangle w_{1}\\
&+\mu^{-2}\langle\nabla^NV,\Phi_{j-1}\rangle \,w_{0}
+\mu^{-2}\langle\nabla^NV,\frac{\bar \xi}{\mu}\rangle \,w_{j-1}
+\mu^{-3}\partial^2_{kl}V(\e y,0)\,\Phi_{j-2}^l\,\bar\xi^k\,w_{0}\\
&-p(p-1)w_0^{p-2}w_1w_{j-1}+G_j(\e y,\bar\xi ;
\Phi_0,\cdots,\Phi_{j-3})\\
=&\mu^{-2}\langle\nabla^NV,\Phi_{j-1}\rangle \,w_{0}+\widetilde{\mathfrak{F}}_j
\end{align*}
and 
\begin{align*}
\mathfrak{E}_j
=-p(p-1)w_0^{p-2}w_{j-1}
+\mu^{-2}\langle\nabla^NV,\Phi_{j-2}\rangle +\widetilde{\mathfrak{E}}_j(\e y,\bar\xi ;\Phi_0,\cdots,\Phi_{j-3}),
\end{align*}
where $\mathcal A_j^i$, $\mathcal B_j^{i\ell}$, $\mathcal C_j^i$, $\mathcal D_j^{ab}$ and $\widetilde{\mathfrak{E}}_j$ are smooth bounded functions on their variables. 

Except for $\e \big( -\e^2\mu^{-2}\D_K e+\lambda_0e \big)Z$, the term of order $\e^j$ vanishes if and only if $w_j$ satisfies the equation
\begin{align*}
L_0w_j&=-\mathfrak{F}_j
-\mathfrak{E}_j\,e\,Z
-\mathcal A_j^i(\e y,\bar\xi ; \Phi_0,\cdots,\Phi_{j-3})\,e\,\partial_i Z
-\mathcal B_j^{i\ell}(\e y,\bar\xi ; \Phi_0,\cdots,\Phi_{j-3})\,e\,\partial^2_{i\ell} Z\\
&\quad-\mathcal C_j^i(\e y,\bar\xi ; \Phi_0,\cdots,\Phi_{j-3})\cdot \nabla_Ke\,\partial_i Z
-\mathcal D_j^{ab}(\e y,\bar\xi ; \Phi_0,\cdots,\Phi_{j-3})\,\partial^2_{ab}e\,Z.
\end{align*}
By Freedholm alternative this equation is solvable if and only if the right hand side is $L^2$ orthogonal to the kernel of linearized operator $L_0$. Before  computing the projection against $\partial_s w_0$, let us recall that
$$
w_{j-1}= w_{j-1,1}+\sigma^{-1}\langle H,\Phi_{j-2}\rangle U_0,
$$
where
$w_{j-1,1}\perp\partial_iw_0$ is a function which does not involve $\Phi_{j-2}$. 

As before we look for a solution $w_j$ of the form
$$
w_{j}=w_{j,1}+\sigma^{-1}\langle H,\Phi_{j-1}\rangle U_0,
$$
where $w_{j,1}\perp\partial_i w_0$ solves
\begin{align*}
L_0w_{j,1}&=-\widetilde{\mathfrak{F}}_j
-\mathfrak{E}_j\,e\,Z
-\mathcal A_j^i(\e y,\bar\xi ; \Phi_0,\cdots,\Phi_{j-3})\,e\,\partial_i Z
-\mathcal B_j^{i\ell}(\e y,\bar\xi ; \Phi_0,\cdots,\Phi_{j-3})\,e\,\partial^2_{i\ell} Z\\
&\quad-\mathcal C_j^i(\e y,\bar\xi ; \Phi_0,\cdots,\Phi_{j-3})\cdot \nabla_Ke\,\partial_i Z
-\mathcal D_j^{ab}(\e y,\bar\xi ; \Phi_0,\cdots,\Phi_{j-3})\,\partial^2_{ab}e\,Z.
\end{align*}

Since $j\geq3$, we can write 
\begin{equation*}
\widetilde{\mathfrak{F}}_j
=\widetilde{\mathfrak{F}}_j(\Phi_{j-2})
=S_{j,0}+S_j(\Phi_{j-2}),
\end{equation*}
where $S_{j,0}=S_{j,0}(\e y,\bar\xi ; \Phi_0,\cdots,\Phi_{j-3})$ does not involve $\Phi_{j-2}$, and $S_j(\Phi_{j-2})$ is the sum of linear terms of $\Phi_{j-2}$. Since 
\begin{align}
\int_{\R^N} S_j(\Phi_{j-2})\,\partial_s w_0&=\mu^{-1}\big(\int_{\mathbb{R}^N}|\partial_1w_0|^2\big)(\mathcal{J}_K\Phi_{j-2})^s(\e y),
\end{align}
the equation on $w_{j,1}$ (and then on $w_j$) is solvable if and only if $\Phi_{j-2}$ satisfies an equation of the form
$$
\mathcal{J}_K\Phi_{j-2}
=\mathfrak H_j(\bar y ; \Phi_0,\cdots,\Phi_{j-3}, e).
$$
This latter equation is solvable by the non-degeneracy condition on $K$. Moreover, for any given $e$, by induction method one can get $\Phi_{j-2}=\Phi_{j-2}(\bar y;e)$ is a smooth bounded function on $\bar y$ and is Lipschitz continuous with respect to $e$. When this is done, since the right hand side of equation of $w_{j,1}$ is a smooth bounded function of $(\e y,\bar\xi)$, we see at once that $w_{j,1}=w_{j,1}(\e y,\bar\xi)$ is a smooth bounded function of $(\e y,\bar\xi)$. Furthermore, $w_{j,1}=w_{j,1}(\e y,\bar\xi;e)$ is Lipschitz continuous with respect to $e$. 

\begin{remark}
To get the higher order approximations, our argument only need the expansion of the Laplace-Beltrami operator up to second order. It is slightly different from the argument used in \cite{wwy}.
\end{remark}

\subsection{Summary}

Let $v_{I}$ be the local approximate solution constructed in the previous section, i.e., 
\begin{align}
v_{I}(y,\bar \xi)
=w_0(\bar \xi)
+\sum_{\ell=1}^I\e^\ell w_\ell(\e y,\bar\xi)
+\e e(\e y) Z(\bar \xi),
\end{align}
for $I\in \N_+$ an arbitrary positive integer.

From the analysis in the previous subsections, the stationary and non-degeneracy conditions on $K$ can be seen as conditions such that $v_I$ is very close to a genuine solution and can be reformulated as follows.
\begin{proposition}
Let $K^k$ be a closed (embedded or immersed) submanifold of $M^n$. Then the stationary condition on $K$ is \eqref{stationary}, and the non-degeneracy condition on $K$ is equivalent to the invertibility of operator $\mathcal{J}_K$ defined in \eqref{J}.
\end{proposition}

Summarizing, we have the following proposition by taking $j=I+1$, $w_{I+1}=0$, and $\Phi_{I+1}=0$ in Section 3.2.3.

\begin{proposition}\label{prop-local approximation} 
Let $I\geq3$ be an arbitrary positive integer, for any given smooth functions $\Phi_{I-1}$ and $e$ on $K$, there are smooth bounded functions 
\[
w_\ell=w_{\ell,1}(\e y,\bar\xi;e)+\sigma^{-1}\langle H,\Phi_{\ell-1}\rangle U_0, 
\quad\ell=1,\dots,I,
\] 
and
\[
\Phi_j=\Phi_j(\bar y; e),
\quad j=0,\dots,I-2,
\]
such that 
\begin{align}
\widetilde S_\e(v_{I})
&=\e\big(-\e^2\mu^{-2}\D_K e+\lambda_0e \big)Z
+\e^{I+1}\widetilde{\mathfrak{F}}_{I+1}
+\e^{I+1}\mathfrak{E}_{I+1}\,e\,Z\nonumber\\[3mm]
&\quad+\e^{I+1}\mathcal A_{I+1}^i(\e y,\bar\xi;e)\,e\,\partial_i Z
+\e^{I+1}\mathcal B_{I+1}^{i\ell}(\e y,\bar\xi;e)\,e\,\partial^2_{i\ell} Z\\[3mm]
&\quad+\e^{I+1}\mathcal C_{I+1}^i(\e y,\bar\xi;e)\cdot \nabla_Ke\,\partial_i Z
+\e^{I+1}\mathcal D_{I+1}^{ab}(\e y,\bar\xi;e)\,\partial^2_{ab}e\,Z
+\mathcal{O}(\e^{I+2}),\nonumber
\end{align}
where 
\begin{align*}
\widetilde{\mathfrak{F}}_{I+1}
=&\mu^{-1}\Gamma_{bl}^b\,\partial_lw_{I}
+\frac{2}{3}\mu^{-1}R_{kisl}\,\bar\xi^k\,\Phi_{I-1}^l\,\partial_{is}^2w_0
-\mu^{-1}\Big(\widetilde g^{ab}\,R_{kabs}+\frac{2}{3}R_{kiis}-\Gamma_{ak}^c\Gamma_{cs}^a\Big)\Phi_{I-1}^k\,\partial_sw_0\\
&+\mu^{-2}\Big(2\,\nabla_K\mu\cdot\nabla_K\Phi_{I-1}^s+\mu\,\D_K\Phi_{I-1}^s\Big)\partial_s w_0
+2(h\mu)^{-1}\Big(\nabla_K h\cdot\,\nabla_K\Phi_{I-1}^s\Big)\partial_sw_0\\
&+2\mu^{-2}\Big(\nabla_K\mu\cdot\nabla_K\Phi_{I-1}^s\Big)(\bar\xi^i\,\partial_{is}^2w_0)
+\mu^{-2}\langle\nabla^NV,\Phi_0\rangle w_{I}
+\mu^{-2}\langle\nabla^NV,\Phi_{I-1}\rangle w_{1}\\
&+\mu^{-2}\langle\nabla^NV,\frac{\bar \xi}{\mu}\rangle \,w_{I}
+\mu^{-3}\partial^2_{kl}V(\e y,0)\,\Phi_{I-1}^l\,\bar\xi^k\,w_{0}
-p(p-1)w_0^{p-2}w_1w_{I}+G_{I+1}(\e y,\bar\xi ;
e),\\[3mm]
\mathfrak{E}_{I+1}
=&-p(p-1)w_0^{p-2}w_{I}
+\mu^{-2}\langle\nabla^NV,\Phi_{I-1}\rangle 
+\widetilde{\mathfrak{E}}_{I+1}(\e y,\bar\xi ;
e),
\end{align*}
and $\mathcal A_{I+1}^i$, $\mathcal B_{I+1}^{i\ell}$, $\mathcal C_{I+1}^i$, $\mathcal D_{I+1}^{ab}$, $\widetilde{\mathfrak{E}}_{I+1}$ and $G_{I+1}$ are smooth bounded functions on their variables and are Lipschitz continuous with respect to $e$.
\end{proposition}

\begin{remark}
For example, $\widetilde{\mathfrak{E}}_{I+1}$ involves the term $\mu^{-3}\partial_{kl}^2V(\e y,0)\,\Phi_{I-2}^l\,\bar\xi^k$.
\end{remark}


\subsection{Global approximation}

In the previous sections, some very accurate local approximate solution $v_I$ have been defined. 

Denote
\begin{equation*}
u_I(y,\xi)=h(\varepsilon y)v_I(y,\bar{\xi}),
\end{equation*}
in the Fermi coordinate. Since $K$ is compact, by the definition of Fermi coordinate, there is a constant $\delta>0$ such that the normal coordinate $x$ on $K_\e$ is well defined for $|x|<1000\delta/\e$. 

Now we can simply define our global approximation:
\begin{equation}\label{glob-aprox}
W(z)=\eta_{3\delta}^\varepsilon(x)u_I(y,\xi)
\quad\text{for}\ z\in M_\e,
\end{equation}
where $\eta_{\ell\delta}^\varepsilon(x):=\eta(\frac{\varepsilon|x|}{\ell\delta})$
and $\eta$ is a nonnegative smooth cutoff function such that 
$$\eta(t)=1\quad \hbox{if $|t|<1$}\quad \hbox{ and}\quad  \eta(t)=0 \quad \hbox{if $|t|>2$}.$$
It is easy to see that $W$ has the concentration property as required. Note that $W$ depends on the parameter functions $\Phi_{I-1}$ and $e$, thus we can write $W=W(\ \cdot\ ;\Phi_{I-1},e)$ and define the configuration space of $(\Phi_{I-1},e)$ by
\begin{align}\label{cond-phi-e}
\Lambda:=\left\{(\Phi_{I-1},e)\,\Big|\,
\begin{array}{ll}
\|\Phi_{I-1}\|_{C^{0,\alpha}(K)}+\|\nabla\Phi_{I-1}\|_{C^{0,\alpha}(K)}+\|\nabla^2\Phi_{I-1}\|_{C^{0,\alpha}(K)}\leq1,\\[3mm]
\|e\|_{C^{0,\alpha}(K)}+\e\|\nabla e\|_{C^{0,\alpha}(K)}+\e^2\|\nabla^2 e\|_{C^{0,\alpha}(K)}\leq1
\end{array}
\right\}.
\end{align}
Clearly, the configuration space $\Lambda$ is infinite dimensional. 

For $(\Phi_{I-1},e)\in\Lambda$, it is not difficult to show that for any $0<\tau<1$, there is a positive constant $C$ (independent of $\e$, $\Phi_{I-1}$, $e$) such that 
\begin{equation}\label{estimate-v}
|v_I(y,\bar \xi)|\leq C e^{-\tau|\bar \xi|},
\quad \forall\,(y,\bar \xi)\in K_\e\times\R^N.
\end{equation}


\section{An infinite dimensional reduction and the proof of Theorem~\ref{th:main}}

To construct the solutions stated in Theorem~\ref{th:main}, we will apply the so-called infinite dimensional reduction which can be seen as a generalization of the classical Lyapunov-Schmidt reduction in an infinite dimensional setting. It has been used in many constructions in PDE and geometric analysis. We present it here in a rather simple and synthetic way since it uses many ideas which have been developed by all the different authors working on this subject or on closely related problems. In particular, we are benefited from the ideas and tricks in \cite{dkw,P,wwy}. 



\subsection{Setting-up of the problem}

Given $(\Phi_{I-1},e)\in\Lambda$, we have defined a global approximate solution $W$. an infinite dimensional reduction will be applied to claim that there exist $\Phi_{I-1}$ and $e$ such that a small perturbation of the global approximation $W$ is a genuine solution.

For this purpose, we denote
\begin{align*}
E:=-\Delta_{g}W+V(\varepsilon z)W-W^p,
\end{align*}
\begin{align*}
L_\varepsilon[\phi]:=-\Delta_{g}\phi+V(\varepsilon z)\phi-pW^{p-1}\phi,
\end{align*}
and
\begin{align*}
N(\phi):=-\big[(W+\phi)^p-W^p-pW^{p-1}\phi\big].
\end{align*}
Obviously, $W+\phi$ is a solution of equation \eqref{eq-u} is equivalent to
\begin{align}\label{eq-phi-0}
L_\varepsilon[\phi]+E+N(\phi)=0.
\end{align}

To solve \eqref{eq-phi-0}, we  look for a solution $\phi$ of the form
\begin{equation*}
\phi:=\eta_{3\delta}^\varepsilon\phi^\sharp+\phi^\flat,
\end{equation*}
where $\phi^\flat:M_\varepsilon\rightarrow\mathbb{R}$ and $\phi^\sharp: K_\varepsilon\times\mathbb{R}^N\rightarrow\mathbb{R}$. This nice argument has been used in \cite{dkw,P,wwy} and is called {\it the gluing technique}. It seems rather counterintuitive, but this strategy will make the linear theory of $L_\e$ clear. 

An easy computation shows that 
\begin{equation*}
-L_\varepsilon[\phi]=
\eta_{3\delta}^{\varepsilon}\left(\Delta_g\phi^\sharp-V\phi^\sharp+pW^{p-1}\phi^\sharp\right)
+\Delta_g\phi^\flat-V\phi^\flat+pW^{p-1}\phi^\flat
+(\Delta_g\eta_{3\delta}^\varepsilon)\phi^\sharp+2\nabla_g\eta_{3\delta}^\varepsilon\cdot\nabla_g\phi^\sharp.
\end{equation*}
Therefore, $\phi$ is a solution of \eqref{eq-phi-0} if the pair $(\phi^\flat,\phi^\sharp)$ satisfies the following coupled system:
\begin{equation*}
\begin{cases}
\Delta_g\phi^\flat-V\phi^\flat
=-(\Delta_g\eta_{3\delta}^\varepsilon)\phi^\sharp-2\nabla_g\eta_{3\delta}^\varepsilon\cdot\nabla_g\phi^\sharp
+(1-\eta_\delta^\varepsilon)\Big[E+N(\eta_{3\delta}^\varepsilon\phi^\sharp+\phi^\flat)-pW^{p-1}\phi^\flat\Big],\\[3mm]
\eta_{3\delta}^{\varepsilon}\left(\Delta_g\phi^\sharp-V\phi^\sharp+pW^{p-1}\phi^\sharp\right)
=\eta_\delta^\varepsilon\Big[E+N(\eta_{3\delta}^\varepsilon\phi^\sharp+\phi^\flat)-pW^{p-1}\phi^\flat\Big].
\end{cases}
\end{equation*}

In order to solve the above system,  we first define 
\begin{equation}
L_\e^\flat[\phi^\flat]
:=\Delta_g\phi^\flat-V\phi^\flat
\quad\text{on}\ M_\varepsilon,
\end{equation}
and note that it is a strongly coercive operator thanks to the conditions on the potential $V$, see \eqref{cond-V}. Then, in the support of $\eta_{3\delta}^\varepsilon$, we define
\begin{equation*}
\phi^\sharp:=h(\varepsilon y)\phi^*(y,\bar{\xi}), \qquad \hbox{with \quad  $\phi^*: K_\varepsilon\times\mathbb{R}^N\rightarrow\mathbb{R}$}.
\end{equation*}
A straightforward computation as in  Subsection 3.1 yields
\begin{equation*}
\eta_{3\delta}^{\e}\Big(\Delta_g\phi^\sharp-V\phi^\sharp+pW^{p-1}\phi^\sharp\Big)
=\eta_{3\delta}^{\e}h^p\Big(\Delta_{\mathbb{R}^N}\phi^*+\mu^{-2}\Delta_{K_\varepsilon}\phi^*-\phi^*+(\eta_{3\delta}^\varepsilon)^{p-1}pv_I^{p-1}\phi^*+\widetilde{B}[\phi^*]\Big). 
\end{equation*}
where $\widetilde{B}=\mathcal{O}(\e)$ is a linear operator defined in Subsection 3.1. Now we extend the linear operator $\widetilde{B}$ to $K_\e\times\mathbb{R}^N$ and we define 
\begin{align*}
\mathbb{L}_\e[\phi^*]
:=\Delta_{\mathbb{R}^N}\phi^*+\mu^{-2}\Delta_{K_\varepsilon}\phi^*-\phi^*+(\eta_{3\delta}^\varepsilon)^{p-1}pv_I^{p-1}\phi^*+\eta_{6\delta}^\varepsilon\widetilde{B}[\phi^*]
\quad\text{on}\ K_\varepsilon\times\mathbb{R}^N,
\end{align*}
and 
\begin{align*}
L_\e^*[\phi^*]
:=\Delta_{\mathbb{R}^N}\phi^*+\mu^{-2}\Delta_{K_\varepsilon}\phi^*-\phi^*+pw_0^{p-1}\phi^*
=-L_0[\phi^*]+\mu^{-2}\Delta_{K_\varepsilon}\phi^*
\quad\text{on}\ K_\varepsilon\times\mathbb{R}^N.
\end{align*}
Since $\eta_{3\delta}^\varepsilon\cdot\eta_{\delta}^\varepsilon=\eta_{\delta}^\varepsilon$ and $\eta_{3\delta}^\varepsilon\cdot\eta_{6\delta}^\varepsilon=\eta_{3\delta}^\varepsilon$, $\phi$ is a solution of \eqref{eq-phi-0} if the pair $(\phi^\flat,\phi^*)$ solves the following coupled system:
\begin{equation*}\label{gluing sys-2}
\begin{cases}
L_\e^\flat[\phi^\flat]
=-(\Delta_g\eta_{3\delta}^\varepsilon)h\phi^*-2\nabla_g\eta_{3\delta}^\varepsilon\cdot\nabla_g(h\phi^*)
+(1-\eta_\delta^\varepsilon)
\Big[E+N(\eta_{3\delta}^\varepsilon\phi^\sharp+\phi^\flat)-pW^{p-1}\phi^\flat\Big],\\[3mm]
L_\e^*[\phi^*]
=\eta_\delta^\varepsilon\,h^{-p}\Big[E+N(\eta_{3\delta}^\varepsilon h\phi^*+\phi^\flat)-pW^{p-1}\phi^\flat\Big]-(\mathbb{L}_\e-L_\e^*)[\phi^*].
\end{cases}
\end{equation*}

\

\noindent It is easy to check that
\begin{align*}
-(\Delta_g\eta_{3\delta}^\varepsilon)h\phi^*
-2\nabla_g\eta_{3\delta}^\varepsilon\cdot\nabla_g(h\phi^*)
=(1-\eta_\delta^\varepsilon)
\Big[-(\Delta_g\eta_{3\delta}^\varepsilon)h\phi^*
-2\nabla_g\eta_{3\delta}^\varepsilon\cdot\nabla_g(h\phi^*)\Big]
\end{align*}
and 
\begin{align*}
(1-\eta_\delta^\e)
=(1-\eta_\delta^\e)(1-\eta_{\delta/2}^\e). 
\end{align*}
Now, we define 
\begin{align*}
\mathcal N_\e(\phi^\flat,\phi^*,\Phi_{I-1},e)
:=&-(\Delta_g\eta_{3\delta}^\varepsilon)h\phi^*
-2\nabla_g\eta_{3\delta}^\varepsilon\cdot\nabla_g(h\phi^*)\\
&+(1-\eta_{\delta/2}^\varepsilon)
\Big[E+N(\eta_{3\delta}^\varepsilon\phi^\sharp+\phi^\flat)-pW^{p-1}\phi^\flat\Big],
\end{align*}
and
\begin{align*}
\mathcal M_\e(\phi^\flat,\phi^*,\Phi_{I-1},e)
:=\eta_\delta^\varepsilon\,h^{-p}\Big[E+N(\eta_{3\delta}^\varepsilon h\phi^*+\phi^\flat)-pW^{p-1}\phi^\flat\Big]-(\mathbb{L}_\e-L_\e^*)[\phi^*].
\end{align*}
Then $W+\phi$ is a solution of equation \eqref{eq-u} if $(\phi^\flat,\phi^*,\Phi_{I-1},e)$ solves the following system: 
\begin{equation}\label{eq-phi1-phi2}
\begin{cases}
L_\e^\flat[\phi^\flat]
=(1-\eta_\delta^\varepsilon)\,\mathcal  N_\e(\phi^\flat,\phi^*,\Phi_{I-1},e),\\[3mm]
L_\e^*[\phi^*]
=\mathcal  M_\e(\phi^\flat,\phi^*,\Phi_{I-1},e).
\end{cases}
\end{equation}

To solve the above system \eqref{eq-phi1-phi2}, we first study the linear theory : on one hand, since the operator  $L_\e^\flat$ is strongly coercive, then we have the solvability of equation $L_\e^\flat[\phi^\flat]=\psi$.  On the other hand, one can check at once that $L_\e^*$ has bounded kernels, e.g., $\partial_j w_0$, $j=1,\dots,N$. Actually, since $L_0$ has a negative eigenvalue $\lambda_0$ with the corresponding eigenfunction $Z$, there may be more bounded kernels of $L_\e^*$. 

Let $\psi$ be a function defined on $K_\e\times\mathbb{R}^N$, we define $\Pi$ to be the $L^2(d\bar\xi)$-orthogonal projection on $\partial_j w_0$'s and $Z$, namely
\begin{equation}
\Pi[\psi]:=\Big(\Pi_1[\psi],\dots,\Pi_N[\psi],\Pi_{N+1}[\psi]\Big),
\end{equation}
where for $j=1,\dots,N$, 
\[
\Pi_{j}[\psi]:=\frac{1}{c_0}\int_{\mathbb{R}^{N}}\psi(y,\bar{\xi})\,\partial_jw_0(\bar{\xi})\,d\bar{\xi},
\quad\text{with}\ 
c_0=\int_{\mathbb{R}^{N}}|\partial_1w_0|^2\,d\bar{\xi},
\]
and 
\[
\Pi_{N+1}[\psi]:=\int_{\mathbb{R}^{N}}\psi(y,\bar{\xi})\,Z(\bar{\xi})\,d\bar{\xi}.
\]
Let us also denote by $\Pi^\perp$ the orthogonal projection on the orthogonal of $\partial_j w_0$'s and $Z$, namely
\[
\Pi^\perp[\psi]:=\psi-\sum_{j=1}^{N}\Pi_j[\psi]\,\partial_j w_0-\Pi_{N+1}[\psi]\,Z.
\]

With these notations, as in the Lyapunov-Schmidt reduction, solving the system~\eqref{eq-phi1-phi2} amounts to solving the system
\begin{equation}\label{eq-phi1-phi2-1}
\begin{cases}
L_\e^\flat[\phi^\flat]
=(1-\eta_\delta^\varepsilon)\mathcal  N_\e(\phi^\flat,\phi^*,\Phi_{I-1},e),\\[3mm]
L_\e^*[\phi^*]
=\Pi^\perp\Big[\mathcal  M_\e(\phi^\flat,\phi^*,\Phi_{I-1},e)\Big],\\[3mm]
\Pi\Big[\mathcal  M_\e(\phi^\flat,\phi^*,\Phi_{I-1},e)\Big]=0.
\end{cases}
\end{equation}
It is to see that one can write 
\begin{align*}
E
=\eta_{3\delta}^\e\,h^p\,\widetilde S_\e(v_{I})
-(\Delta_g \eta_{3\delta}^\e)(h v_I)
-2(\nabla_g \eta_{3\delta}^\e)\cdot\nabla_g(h v_I)
-\eta_{3\delta}^\e\Big[(\eta_{3\delta}^\e)^{p-1}-1\Big]h^p v_I^p.
\end{align*}
Hence by Proposition~\ref{prop-local approximation}, 
\begin{align*}
\mathcal M_\e(\phi^\flat,\phi^*,\Phi_{I-1},e)
=&\e\big(-\e^2\mu^{-2}\D_K e+\lambda_0e \big)Z
+\e^{I+1}S_{I+1}(\Phi_{I-1})\\
&+\e^{I+1}G_{I+1}(\e y,\bar\xi;e)
+\e^{I+2}J_{I+1}(\e y,\bar\xi;\Phi_{I-1},e)\\
&+\eta_\delta^\varepsilon\,h^{-p}\Big[N(\eta_{3\delta}^\varepsilon h\phi^*+\phi^\flat)-pW^{p-1}\phi^\flat\Big]-(\mathbb{L}_\e-L_\e^*)[\phi^*].
\end{align*}

On the other hand, since 
\begin{align}
\int_{\R^N} S_{I+1}(\Phi_{I-1})\,\partial_s w_0
=c_0\mu^{-1}(\mathcal{J}_K\Phi_{I-1})^s(\e y),
\end{align}
by some rather tedious and technical computations, one can show that 
\begin{align}\label{sys-phi-e}
\Pi\Big[\mathcal  M_\e(\phi^\flat,\phi^*,\Phi_{I-1},e)\Big]=0
\Longleftrightarrow
\left\{
\begin{array}{ll}
\e^{I+1}\mathcal{J}_K[\Phi_{I-1}]
=\e^{I+1}\mathfrak H_{I+1}(\bar y;e)+\mathcal  M_{\varepsilon,1}(\phi^\flat,\phi^*,\Phi_{I-1},e);
\\[3mm]
\e\,\mathcal{K}_\varepsilon[e]
=\mathcal  M_{\varepsilon,2}(\phi^\flat,\phi^*,\Phi_{I-1},e),
\end{array}
\right.
\end{align}
where $\mathfrak H_{I+1}(\bar y;e)$ is a smooth bounded function on $\bar y$ and is Lipschitz continuous with respect to $e$, $\mathcal{J}_K$ is the Jacobi operator on $K$, and $\mathcal{K}_\varepsilon$ is a Schr\"odinger operator defined by
\begin{equation}
\mathcal{K}_\varepsilon[e]
:=-\varepsilon^2\Delta_Ke+\lambda_0\mu^2e
\end{equation}
where $\lambda_0$ is the unique negative eigenvalue of $L_0$. 

We summarize the above discussion by saying that the function
\[
u=W(\ \cdot\ ;\Phi_{I-1},e)+\eta_{3\delta}^\varepsilon\,h\,\phi^*+\phi^\flat,
\]
is a solution of the equation
\[
\Delta_g u-V(\e z)u+u^p=0,
\]
if the functions $\phi^\flat$, $\phi^*$, $\Phi_{I-1}$ and $e$ satisfy the following system
\begin{equation}\label{reduce system-final}
\begin{cases}
L_\e^\flat[\phi^\flat]
=(1-\eta_\delta^\varepsilon)\,\mathcal  N_\e(\phi^\flat,\phi^*,\Phi_{I-1},e),\\[3mm]
L_\e^*[\phi^*]
=\Pi^\perp\Big[\mathcal  M_\e(\phi^\flat,\phi^*,\Phi_{I-1},e)\Big],\\[3mm]
\e^{I+1}\mathcal{J}_K[\Phi_{I-1}]
=\e^{I+1}\mathfrak H_{I+1}(\bar y;e)
+\mathcal  M_{\varepsilon,1}(\phi^\flat,\phi^*,\Phi_{I-1},e),
\\[3mm]
\e\,\mathcal{K}_\varepsilon[e]
=\mathcal  M_{\varepsilon,2}(\phi^\flat,\phi^*,\Phi_{I-1},e).
\end{cases}
\end{equation}

\begin{remark} 
\begin{enumerate}
\item In general there are two different approaches  to set-up the problem: the first one, as used in  \cite{dkw} and \cite{wwy}, consists in solving first the equations of $\phi^\flat$ and $\phi^*$ for fixed $\Phi_{I-1}$ and $e$, and then solve the left equations of $\Phi_{I-1}$ and $e$. The second one, as in \cite{mm,mal} consists in solving first  the linear problem $L_\varepsilon[\phi]+\psi=0$ under some non-degeneracy and gap conditions; and then  solve the nonlinear problem $L_\varepsilon[\phi]+E+N(\phi)=0$ by using a fixed point arguments.  

Our approach is slightly different from those in \cite{dkw}-\cite{wwy}  and \cite{mm}-\cite{mal}.\\

\item After solving the system~\eqref{reduce system-final}, one can prove the positivity of $u$ by contradiction since both $\phi^\flat$ and $\phi^*$ are small. 
\end{enumerate}


\end{remark}



\subsection{Analysis of the linear operators}

By the above analysis, what is left is to show that \eqref{reduce system-final} has a solution. To this end, we will apply a fixed point theorem. Before we do this, a linear theory will be developed. 

\subsubsection{Analysis of a strongly coercive operator}


To deal with the term $-\eta_\delta^\varepsilon\,h^{-p}pW^{p-1}\phi^\flat$ in $\mathcal M_\e(\phi^\flat,\phi^*,\Phi_{I-1},e)$ in applying a fixed point theorem, one needs to choose norms with the property that $\mathcal M_\e(\phi^\flat,\phi^*,\Phi_{I-1},e)$ depends slowly on $\phi^\flat$. To this end, we define 
\begin{equation}
\|\phi^\flat\|_{\e,\infty}=\|(1-\eta_{\delta/4}^\e)\phi^\flat\|_\infty+\frac{1}{\e}\|\eta_{\delta/4}^\e \phi^\flat\|_\infty.
\end{equation}
With this notation, by the exponential decay of $W$, we have 
\begin{equation*}
\|\mathcal M_\e(\phi^\flat,\phi^*,\Phi_{I-1},e)\|_\infty
\leq C\e\|\phi^\flat\|_{\e,\infty}
\end{equation*}
and
\begin{equation*}
\|\mathcal M_\e(\phi^\flat_1,\phi^*,\Phi_{I-1},e)-\mathcal M_\e(\phi^\flat_2,\phi^*,\Phi_{I-1},e)\|_\infty
\leq C\e\|\phi^\flat_1-\phi^\flat_2\|_{\e,\infty}.
\end{equation*}

Since \eqref{cond-V}, we have the following lemma. 
\begin{lemma}
For any function $\psi(z)\in L^\infty(M_\e)$, there is a unique bounded solution $\phi$ of
\begin{equation}
L_\e^\flat[\phi]
=(1-\eta_{\delta}^\varepsilon)\psi.
\end{equation}
Moreover, there exists a constant $C>0$ (independent of $\varepsilon$) such that
\begin{equation}
\|\phi\|_{\e,\infty}\leq C\|\psi\|_\infty.
\end{equation}
\end{lemma}

For $\phi^\flat\in C_0^{0,\alpha}(M_\e)$, we define 
\begin{equation}
\|\phi^\flat\|_{\e,\alpha}=\|(1-\eta_{\delta/4}^\e)\phi^\flat\|_{C_0^{0,\alpha}}+\frac{1}{\e}\|\eta_{\delta/4}^\e \phi^\flat\|_{C_0^{0,\alpha}}.
\end{equation}
As a consequence of standard elliptic estimates, the following lemma holds. 
\begin{lemma}
For any function $\psi\in C_0^{0,\alpha}(M_\e)$, there is a unique solution $\phi\in C_0^{2,\alpha}(M_\e)$ of
\begin{equation}
L_\e^\flat[\phi]
=(1-\eta_{\delta}^\varepsilon)\psi.
\end{equation}
Moreover, there exists a constant $C>0$ (independent of $\varepsilon$) such that
\begin{equation}
\|\phi\|_{2,\e,\alpha}:=\|\phi\|_{\e,\alpha}+\|\nabla\phi\|_{\e,\alpha}
+\|\nabla^2\phi\|_{\e,\alpha}\leq C\|\psi\|_{C_0^{2,\alpha}(M_\e)}.
\end{equation}
\end{lemma}

\subsubsection{Study of the model linear operator $L_\e^*$}

First, we will prove an injectivity result which is the key result. Then, we will use this result to obtain an a priori estimate and the existence result for solutions of $L_\e^*[\phi]=\psi$ when $\Pi[\phi]=0=\Pi[\psi]$. 

\begin{lemma}[The injectivity result]\label{lemma-injectivity}
Suppose that $\phi\in L^\infty(K_\varepsilon\times \mathbb{R}^{N})$ satisfies $L_\e^*[\phi]=0$ and $\Pi[\phi]=0$.
Then $\phi\equiv0$.
\end{lemma}
\begin{proof}
We will prove this lemma by two steps.
\vspace{2mm}

{\bf Step 1:} The function $\phi(y,\bar{\xi})$ decays exponentially in the variables $\bar{\xi}$.

To prove this fact, it suffices to apply the maximum principle since $w_0(\bar{\xi})$ has exponential decay and $\phi$ is bounded.

\vspace{2mm}

{\bf Step 2:} We next prove that 
\[
f(y)
:=\int_{\mathbb{R}^{N}}\phi^2(y,\bar{\xi})\,d\bar{\xi}
=0,\quad\forall\,y\in K_\varepsilon.
\]

Indeed, by Step 1, for all $y\in K_\varepsilon$, $f(y)$ is well defined. Since $L_\e^*[\phi]=0$, we have
\begin{align*}
\Delta_{K_\varepsilon}f
&=\int_{\mathbb{R}^{N}}2\phi\Delta_{K_\varepsilon}\phi\,d\bar{\xi}
+\int_{\mathbb{R}^{N}}2|\nabla_{K_\varepsilon}\phi|^2\,d\bar{\xi}\\
&=2\mu^2\int_{\mathbb{R}^{N}}\Big\{|\nabla_{\bar{\xi}}\phi|^2+\phi^2-pw_0^{p-1}\phi^2\Big\}\,d\bar{\xi}+2\int_{\mathbb{R}^{N}}|\nabla_{K_\varepsilon}\phi|^2\,d\bar{\xi}\\
&\geq 2\mu^2\gamma_0\int_{\mathbb{R}^{N}}\phi^2(y,\bar{\xi})\,d\bar{\xi},
\end{align*}
where in the last inequality since $\Pi[\phi]=0$ we use the following inequality
\begin{equation}\label{inequality-orth}
\int_{\mathbb{R}^{N}}\Big\{|\nabla_{\bar{\xi}}\phi|^2+\phi^2-pw_0^{p-1}\phi^2\Big\}\,d\bar{\xi}\geq\gamma_0\int_{\mathbb{R}^{N}}\phi^2\,d\bar{\xi}.
\end{equation}

Therefore, by the definition of $f$, the above inequality gives
\begin{equation*}
\Delta_{K_\varepsilon}f\geq 2\mu^2\gamma_0f.
\end{equation*}
Since $f$ is nonnegative and $K_\varepsilon$ is compact, we just get $f\equiv0$ by the integration. If $K_\varepsilon$ is non compact, one can first show that $f$ goes to zero at infinity by the comparison theorem and then get $f\equiv0$ by the maximum principle.
\end{proof}

\begin{remark}
Actually, following the argument of proof of Lemma 3.7 in \cite{P}, one can show that 
\begin{align}
\phi
=\sum_{j=1}^N c^j(y)\,\partial_j w_0
+c^{N+1}(y)\, Z,
\end{align}
if $\phi$ is a bounded solution of $L_\e^*[\phi]=0$, where $c_j(y)$ ($j=1,\dots,N$) can be any bounded function, but $c^{N+1}(y)$ must satisfy the equation 
\begin{equation}\label{eq-c-z}
\Delta_{K_\e}c^{N+1}=\lambda_0\mu^2c^{N+1}.
\end{equation}
It is worth noting that \eqref{eq-c-z} is just another form of $\mathcal{K}_\varepsilon[e]=0$. When $\e$ satisfies some gap condition (cf. Proposition~\ref{prop-gap} below), equation~\eqref{eq-c-z} does not have a bounded solution. 
\end{remark}

Moreover, one can show that under the orthogonal conditions $\Pi[\phi]=0$, the linear operator $L_\e^*$ has only negative eigenvalues $\lambda_j^\varepsilon$'s and there exists a constant $c_0$ such that
\begin{equation*}
\lambda_j^\varepsilon\leq -c_0<0.
\end{equation*}
To prove it, since $\mu^2=V(\bar{y},0)$ and \eqref{cond-V}, the inequality \eqref{inequality-orth} implies 
\begin{equation*}
\int_{K_\varepsilon\times \mathbb{R}^{N}}-L_\e^*[\phi]\,\phi
\geq c\int_{K_\varepsilon\times \mathbb{R}^{N}}(-L_\e^*[\phi])(\mu^2\phi)
\geq c\gamma_0
\int_{K_\varepsilon\times \mathbb{R}^{N}}\phi^2.
\end{equation*}

Before stating the surjectivity result, we define 
\begin{align*}
\|\psi\|_{\e,\alpha,\rho}:=\sup_{(y,\bar{\xi})\in K_\varepsilon\times \mathbb{R}^{N}}e^{\rho|\bar{\xi}|}\|\psi\|_{C^{0,\alpha}(B_1((y,\bar{\xi})))},
\end{align*}
where $\alpha$  and $\rho$ are  small positive constants.

\begin{proposition}[The surjectivity result]
For any function $\psi$ with $\|\psi\|_{\alpha,\sigma}<\infty$ and $\Pi[\psi]=0$, the problem
\begin{equation}
L_\e^*[\phi]=\psi
\end{equation}
has a unique solution $\phi$ with $\Pi[\phi]=0$. Moreover, the following estimate holds: 
\begin{align}
\|\phi\|_{2,\e,\alpha,\rho}
:=\|\phi\|_{\e,\alpha,\rho}+\|\nabla\phi\|_{\e,\alpha,\rho}+\|\nabla^2\phi\|_{\e,\alpha,\rho}
\leq C\|\psi\|_{\e,\alpha,\rho},
\end{align}
where $C$ is a constant independent of $\e$. 
\end{proposition}

\begin{remark}
Here we choose to use weighted H\"{o}lder norms, actually one can also use weighted Sobolev norms.
\end{remark}

\subsubsection{Non-degeneracy condition and invertibility of $\mathcal{J}_K$}

\begin{proposition}\label{p:Jacobi}
Suppose that $K$ is non-degenerate, then for any $\Psi\in (C^{0,\alpha}(K))^{N}\cap NK$, there exists a unique $\Phi\in (C^{2,\alpha}(K))^{N}\cap NK$ such that
\begin{equation}
\mathcal{J}_K[\Phi]=\Psi
\end{equation}
with the property
\begin{equation}
\|\Phi\|_{2,\alpha}
:=\|\Phi\|_{C^{0,\alpha}(K)}+\|\nabla\Phi\|_{C^{0,\alpha}(K)}+\|\nabla^2\Phi\|_{C^{0,\alpha}(K)}\leq C\|\Psi\|_{C^{0,\alpha}(K)},
\end{equation}
where $C$ is a positive constant depending only on $K$.
\end{proposition}
\begin{proof}
Since the Jacobi operator $\mathcal{J}_K$ is self-adjoint, this result follows from the standard elliptic estimates, cf. \cite{GT,Lady}.
\end{proof}


\subsubsection{Gap condition and invertibility of $\mathcal{K}_\varepsilon$}

\begin{proposition}\label{prop-gap}
There is a sequence $\e=\varepsilon_j\searrow0$ such that for any $\varphi\in C^{0,\alpha}(K)$, there exists a unique $e\in C^{2,\alpha}(K)$ such that
\begin{equation}\label{eq-K-e}
\mathcal{K}_{\varepsilon}[e]=\varphi
\end{equation}
with the property
\begin{equation}\label{estimate-e}
\|e\|_*
:=\|e\|_{C^{0,\alpha}(K)}+\varepsilon\|\nabla e\|_{C^{0,\alpha}(K)}+\varepsilon^2\|\nabla^2e\|_{C^{0,\alpha}(K)}
\leq C\varepsilon^{-3k}\|\varphi\|_{C^{0,\alpha}(K)},
\end{equation}
where $C$ is a positive constant independent of $\varepsilon_j$.
\end{proposition}

\begin{proof}
This is a semiclassical analysis of a Schr\"{o}dinger operator. The arguments are  similar in spirit as the ones used in the proof of Proposition 8.1 in \cite{wwy}.  We summarize them in the following two steps. 

\noindent{\bf Step 1:} There is a sequence $\varepsilon_j\searrow0$ such that for any $\varphi\in L^{2}(K)$, there exists a unique solution to \eqref{eq-K-e} and satisfies
\begin{equation}
\|e\|_{L^2(K)}\leq C\varepsilon_j^{-k}\|\varphi\|_{L^2(K)}.
\end{equation}

This fact follows from the variational characterisation of the eigenvalues and the Weyl's asymptotic formula.

\

\noindent{\bf Step 2:} The unique solution satisfies \eqref{estimate-e}. This  follows  from  standard elliptic estimates and Sobolev embedding theorem.
\end{proof}

\subsection{The nonlinear scheme}

Now we can develop the nonlinear theory and complete the proof of Theorem~\ref{th:main}. 

\subsubsection{Size of the error}

\begin{lemma}\label{lem4.4}
There is a constant $C$ independent of $\e$ such that the following estimates hold:
\begin{equation}
\big\|\mathcal  N_\e(0,0,0,0)\big\|_{C_0^{2,\alpha}(M_\e)}
+\big\|\Pi^\perp\big[\mathcal  M_\e(0,0,0,0)\big]\big\|_{\e,\alpha,\rho}
\leq C\e^{I+1}.
\end{equation}
Moreover,
\begin{equation}
\big\|\mathcal  M_{\e,1}(0,0,0,0)\big\|_{C^{0,\alpha}(K)}
\leq C\e^{I+2},
\quad
\big\|\mathcal  M_{\e,2}(0,0,0,0)\big\|_{C^{0,\alpha}(K)}
\leq C\e^{I+1}.
\end{equation}
\end{lemma}

\begin{proof}
It follows from the definitions and the estimate  \eqref{estimate-v}.
\end{proof}

\subsubsection{Lipschitz continuity} 

According to the estimate of error, we define
\begin{align}
\begin{array}{ll}
\mathcal{B}_\lambda
:=\bigg\{(\phi^\flat,\phi^*,\Phi_{I-1},e)\,\big|
&\|\phi^\flat\|_{2,\e,\alpha}\leq \lambda \e^{I+1},
\|\phi^*\|_{2,\e,\alpha,\rho}\leq \lambda \e^{I+1},\\
&\|\Phi_{I-1}\|_{2,\alpha}\leq \lambda \e,
\|e\|_*\leq \lambda \e^{I-3k}
\bigg\}.
\end{array}
\end{align}

\begin{lemma}\label{lem4.5}
Given $(\phi^\flat_1,\phi^*_1,\Phi_{I-1},e_1),(\phi^\flat_2,\phi^*_2,\widetilde\Phi_{I-1},e_2)\in \mathcal{B}_\lambda$, 
there is a constant $C$ independent of $\e$ such that the following estimates hold:
\begin{align*}
&\big\|\mathcal  N_\e(\phi^\flat_1,\phi^*_1,\Phi_{I-1},e_1)
-\mathcal  N_\e(\phi^\flat_2,\phi^*_2,\widetilde\Phi_{I-1},e_2)\big\|_{C_0^{2,\alpha}(M_\e)}\\
&\leq
C\e^{I+1} \Big(\|\phi^\flat_1-\phi^\flat_2\|_{2,\e,\alpha}
+\|\phi^*_1-\phi^*_2\|_{2,\e,\alpha,\rho}
+\|\Phi_{I-1}-\widetilde\Phi_{I-1}\|_{2,\alpha}
+\|e_1-e_2\|_{*}\Big),
\end{align*}

\begin{align*}
&\big\|\Pi^\perp\big[\mathcal  M_\e(\phi^\flat_1,\phi^*_1,\Phi_{I-1},e_1)\big]
-\Pi^\perp\big[\mathcal  M_\e(\phi^\flat_2,\phi^*_2,\widetilde\Phi_{I-1},e_2)\big]\big\|_{\e,\alpha,\rho}\\
&\leq
C\e^{I+1}\Big(\|\phi^\flat_1-\phi^\flat_2\|_{2,\e,\alpha}
+\|\phi^*_1-\phi^*_2\|_{2,\e,\alpha,\rho}
+\|\Phi_{I-1}-\widetilde\Phi_{I-1}\|_{2,\alpha}
+\|e_1-e_2\|_{*}\Big),
\end{align*}

\begin{align*}
&\big\|\mathcal  M_{\e,1}(\phi^\flat_1,\phi^*_1,\Phi_{I-1},e_1)
-\mathcal  M_{\e,1}(\phi^\flat_2,\phi^*_2,\widetilde\Phi_{I-1},e_2)\big\|_{C^{0,\alpha}(K)}\\
&\leq
C\e^{I+2}\Big(\|\phi^\flat_1-\phi^\flat_2\|_{2,\e,\alpha}
+\|\phi^*_1-\phi^*_2\|_{2,\e,\alpha,\rho}
+\|\Phi_{I-1}-\widetilde\Phi_{I-1}\|_{2,\alpha}
+\|e_1-e_2\|_{*}\Big),
\end{align*}
and
\begin{align*}
&\big\|\mathcal  M_{\e,2}(\phi^\flat_1,\phi^*_1,\Phi_{I-1},e_1)-\mathcal  M_{\e,2}(\phi^\flat_2,\phi^*_2,\widetilde\Phi_{I-1},e_2)\big\|_{C^{0,\alpha}(K)}\\
&\leq
C\e^{I+1}\Big(\|\phi^\flat_1-\phi^\flat_2\|_{2,\e,\alpha}
+\|\phi^*_1-\phi^*_2\|_{2,\e,\alpha,\rho}
+\|\Phi_{I-1}-\widetilde\Phi_{I-1}\|_{2,\alpha}
+\|e_1-e_2\|_{*}\Big).
\end{align*}
\end{lemma}

\begin{proof}
This proof is rather technical but does not offer any real difficulty. It is worth noting that the use of  the norm $\|\phi^\flat\|_{2,\e,\alpha}$ is crucial to estimate the term $-\eta_\delta^\varepsilon\,h^{-p}pW^{p-1}\phi^\flat$ in $\mathcal M_\e(\phi^\flat,\phi^*,\Phi_{I-1},e)$.
\end{proof}

\subsubsection{Proof of Theorem~\ref{th:main}}

By the analysis in Section 4.1, the proof of Theorem~\ref{th:main} follows from the solvability of \eqref{reduce system-final}. 

Now we can use the results in the linear theory to rephrase the solvability of \eqref{reduce system-final} as a fixed point problem. To do this, let $\Phi_{I-1}=\Phi_{I-1,0}+\widetilde\Phi_{I-1}$, where $\Phi_{I-1,0}$ solve the equation 
\begin{equation}
\mathcal{J}_K[\Phi_{I-1,0}]=\mathfrak H_{I+1}(\bar y;e).
\end{equation}
Thus $\Phi_{I-1,0}=\Phi_{I-1,0}(\bar y;e)$. Moreover, the reduced system~\eqref{reduce system-final} becomes 
\begin{equation}\label{reduce system-final-new}
\begin{cases}
L_\e^\flat[\phi^\flat]
=(1-\eta_\delta^\varepsilon)\,\mathcal N_\e(\phi^\flat,\phi^*,\Phi_{I-1},e),\\[3mm]
L_\e^*[\phi^*]
=\Pi^\perp\Big[\mathcal  M_\e(\phi^\flat,\phi^*,\Phi_{I-1},e)\Big],\\[3mm]
\e^{I+1}\mathcal{J}_K[\widetilde\Phi_{I-1}]
=\widetilde{\mathcal  M}_{\varepsilon,1}(\phi^\flat,\phi^*,\widetilde\Phi_{I-1},e),
\\[3mm]
\e\,\mathcal{K}_\varepsilon[e]
=\widetilde{\mathcal  M}_{\varepsilon,2}(\phi^\flat,\phi^*,\widetilde\Phi_{I-1},e).
\end{cases}
\end{equation}

It is a simple matter to check that both $\widetilde{\mathcal  M}_{\varepsilon,1}$ and $\widetilde{\mathcal  M}_{\varepsilon,2}$ satisfy the properties in Lemmas~\ref{lem4.4} and \ref{lem4.5}. Taking $I\geq3k+1$ and $\lambda$ sufficiently large, Theorem~\ref{th:main} is now a simple consequence of a fixed point theorem for contraction mapping in $\mathcal{B}_\lambda$.

\

\section{Appendix: Proof of Proposition \ref{expansion-g}}\label{s:apx}

The proof is based on the Taylor expansion of the metric coefficients. We recall that the Laplace-Beltrami operator is given by
$$
\Delta_g u= \frac1{\sqrt{\det g}}\,\partial_\alpha\bigg( \sqrt{\det g} \,g^{\alpha\beta}\, \partial_\beta u\bigg)
$$
which can be rewritten as
\begin{align*}
\Delta_{g}u=g^{\alpha\beta}\partial_{\alpha\beta}^2u+(\partial_\alpha g^{\alpha\beta})\partial_\beta u+\frac{1}{2}\,g^{\alpha\beta}\partial_\alpha(\log\det g)\partial_\beta u.
\end{align*}
Using the expansion of the metric coefficients determined above, we can easily prove that
\begin{align*}
&g^{\alpha\beta}\,\partial_{\alpha\beta}^2u
=\widetilde g^{ab}\,\partial_{ab}^2u+\partial^2_{ii}u+\e\,\bigg\{\widetilde g^{cb}\,\Gamma_{ci}^a+\widetilde g^{ca}\,\Gamma_{ci}^b\bigg\}\,(\xi^i+\Phi^i)\,\widetilde g^{ab}\,\partial_{ab}^2u-2\varepsilon\,\widetilde g^{ab}\,\partial_{\bar b}\Phi^j\,\partial_{aj}^2u\\
&+\varepsilon^2\bigg(-\widetilde g^{cb}\,\widetilde g^{ad}\,R_{kcdl}+\widetilde g^{ac}\,\Gamma_{dk}^b\Gamma_{cl}^d+\widetilde g^{bc}\,\Gamma_{dk}^a \Gamma_{cl}^d
+\widetilde g^{cd}\,\Gamma_{dk}^a \Gamma_{cl}^b\bigg)\,(\xi^k+\Phi^k)(\xi^l+\Phi^l)\,\partial_{ab}^2u\\
&-\frac{4\,\e^2}{3}R_{kajl}(\xi^k+\Phi^k)(\xi^l+\Phi^l)\,\partial_{aj}^2u+2
\varepsilon^2\partial_{\bar b}\Phi^j\,\bigg\{\widetilde g^{bc}\,\Gamma_{ci}^a+\widetilde g^{ac}\,\Gamma_{ci}^b\bigg\}\,(\xi^i+\Phi^i)\,\partial_{aj}^2u\\
&-\frac{\e^2}{3}\,R_{kijl}(\xi^k+\Phi^k)(\xi^l+\Phi^l)\,\partial^2_{ij}u+\varepsilon^2\,\widetilde g^{ab}\,\partial_{\bar a}\Phi^i\partial_{\bar b}\Phi^j\,\partial^2_{ij}u\\
&+R_3(\xi,\Phi,\nabla\Phi)(\partial_{ij}^2u+\partial_{aj}^2u+\partial_{ab}^2u).
\end{align*}
An  easy computations yields 
\begin{align*}
\partial_b{g}^{ab}&=\partial_b\widetilde g^{ab}+\e^2\,\partial_{\bar b}\bigg\{\widetilde g^{cb}\,\Gamma_{ci}^a+\widetilde g^{ca}\,\Gamma_{ci}^b\bigg\}\,(\xi^i+\Phi^i)+\e^2\,\bigg\{\widetilde g^{cb}\,\Gamma_{ci}^a+\widetilde g^{ca}\,\Gamma_{ci}^b\bigg\}\,\partial_{\bar b}\,\Phi^i\\
&\quad+R_3(\xi,\Phi,\nabla\Phi,\nabla^2\Phi),\\[3mm]
\partial_jg^{ja}&=-\frac{2}{3}\varepsilon^2R_{jajl}(\xi^l+\Phi^l)+
\varepsilon^2\partial_{\bar b}\Phi^j\,\bigg\{\widetilde g^{bc}\,\Gamma_{cj}^a+\widetilde g^{ac}\,\Gamma_{cj}^b\bigg\}+R_3(\xi,\Phi,\nabla\Phi),\\[3mm]
\partial_ag^{aj}&=-\e^2\,\partial_{\bar a}\widetilde g^{ab}\,\partial_{\bar b}\Phi^j- \e^2\,\widetilde g^{ab}\,\partial^2_{\bar a\bar b}\Phi^j+
\varepsilon^3\partial^2_{\bar a\bar b}\Phi^j\,\bigg\{\widetilde g^{bc}\,\Gamma_{ci}^a+\widetilde g^{ac}\,\Gamma_{ci}^b\bigg\}\,(\xi^i+\Phi^i)\\
&\quad+R_3(\xi,\Phi,\nabla\Phi,\nabla^2\Phi),\\\\[3mm]
\partial_ig^{ij}&=-\frac{1}{3}\varepsilon^2R_{kiji}(\xi^k+\Phi^k)+R_3(\xi,\Phi,\nabla\Phi).
\end{align*}
Then the following expansion holds
\begin{align*}
&(\partial_\alpha g^{\alpha\beta})\partial_\beta u=\\&\partial_b\widetilde g^{ab}\,\partial_au+\e^2\,\partial_{\bar b}\bigg\{\widetilde g^{cb}\,\Gamma_{ci}^a+\widetilde g^{ca}\,\Gamma_{ci}^b\bigg\}\,(\xi^i+\Phi^i)\,\partial_au+\e^2\,\bigg\{\widetilde g^{cb}\,\Gamma_{ci}^a+\widetilde g^{ca}\,\Gamma_{ci}^b\bigg\}\,\partial_{\bar b}\,\Phi^i\,\partial_au\\
&-\frac{2}{3}\varepsilon^2R_{jajl}(\xi^l+\Phi^l)\,\partial_au+
\varepsilon^2\partial_{\bar b}\Phi^j\,\bigg\{\widetilde g^{bc}\,\Gamma_{cj}^a+\widetilde g^{ac}\,\Gamma_{cj}^b\bigg\}\,\partial_au\\
&-\e^2\,\partial_{\bar a}\widetilde g^{ab}\,\partial_{\bar b}\Phi^j\,\partial_ju- \e^2\,\widetilde g^{ab}\,\partial^2_{\bar a\bar b}\Phi^j\,\partial_ju+
\varepsilon^3\partial^2_{\bar a\bar b}\Phi^j\,\bigg\{\widetilde g^{bc}\,\Gamma_{ci}^a+\widetilde g^{ac}\,\Gamma_{ci}^b\bigg\}\,(\xi^i+\Phi^i)\,\partial_ju\\
&-\frac{1}{3}\varepsilon^2R_{kiji}(\xi^k+\Phi^k)\,\partial_ju+R_3(\xi,\Phi,\nabla\Phi,\nabla^2\Phi)(\partial_ju+\partial_au).
\end{align*}
On the other hand using the expansion of the log of determinant of $g$ given in Lemma \ref{lem:inverse-det}, we obtain
\begin{align*}
\partial_b \log\big(\det g\big)=\partial_b\log\big(\det \widetilde g\big)-2\e^2\,\partial_{\bar b}\big(\Gamma^a_{ak}\big)\,(\xi^k+\Phi^k)-2\e^2\,\Gamma^a_{ak}\,\partial_{\bar b}\Phi^k+R_3(\xi,\Phi,\nabla\Phi,\nabla^2\Phi).\\
\end{align*}
and
\begin{align*}
\partial_i(\log\det g)=-2\e\,\Gamma^b_{bi}+2\varepsilon^2\bigg(\widetilde g^{ab}\,R_{kabi}+\frac{1}{3}R_{kjji}-\Gamma_{ak}^c\Gamma_{ci}^a\bigg)(\xi^k+\Phi^k)+R_3(\xi,\Phi,\nabla\Phi),
\end{align*}
which implies that
\begin{align*}
&\frac{1}{2}g^{\alpha\beta}\partial_\alpha(\log\det g)\partial_\beta u=\\&\frac{1}{2}\,\partial_a(\log\det \widetilde{g})\,\bigg(
\widetilde g^{ab}\,\partial_bu+\e\big\{\widetilde g^{cb}\Gamma_{ci}^a+\widetilde g^{ca}\Gamma_{ci}^b\big\}(\xi^i+\Phi^i)\partial_bu-\varepsilon\,\widetilde g^{ab}\partial_{\bar b}\Phi^j\partial_ju\bigg)\\
&-\e\,\Gamma^b_{bi}\partial_iu+\varepsilon^2\bigg(\widetilde g^{ab}R_{kabi}+\frac{1}{3}R_{kjji}-\Gamma_{ak}^c\Gamma_{ci}^a\bigg)(\xi^k+\Phi^k)\,\partial_i u\\
&-\e^2\,\bigg( \,\partial_{\bar b}\big(\Gamma^d_{dk}\big)\,(\xi^k+\Phi^k)+\Gamma^d_{dk}\,\partial_{\bar b}\Phi^k \bigg)\,\widetilde g^{ab}\partial_au+R_3(\xi,\Phi,\nabla\Phi,\nabla^2\Phi)(\partial_ju+\partial_au).
\end{align*}
Collecting the above terms and recalling that
\begin{align*}
\Delta_{K_\e}u=\widetilde g^{ab}\partial_{ab}^2u+(\partial_a \widetilde g^{ab})\partial_b u+\frac{1}{2}\,\widetilde g^{ab}\partial_a(\log\det \widetilde g)\partial_b u,
\end{align*}
the desired result then follows at once.

\

\noindent  {\it Acknowledgments.}  F. Mahmoudi has been supported by  Fondecyt Grant 1140311
and Fondo Basal CMM. W. Yao is supported by Fondecyt Grant 3130543. The authors would like  to thank  Martin Man-chun Li for helpful conversations and discussions. 

\

\end{document}